\documentclass[pdflatex,sn-mathphys-num]{sn-jnl}

\usepackage{amssymb}
\usepackage{amsmath}
\usepackage{amsthm}
\usepackage{yhmath}
\usepackage{mathdots}
\usepackage{MnSymbol}
\usepackage{a4wide}
\usepackage{color}

\usepackage{enumitem}
\usepackage{nccmath}

\usepackage{blkarray}
\usepackage{multirow}
\usepackage{float}
\usepackage{booktabs}
\usepackage{array}
\usepackage{bbm}

\usepackage{caption, subcaption}
\usepackage{graphicx}

\usepackage{natbib}
\usepackage[english]{babel}

\usepackage{algorithm}
\usepackage{algpseudocode}

\newtheorem{theorem}{\bf Theorem}[section]

\newtheorem{remark}[theorem]{\bf Remark}
\newtheorem{lemma}[theorem]{\bf Lemma}

\SetLabelAlign{CenterWithParen}{\hfil(\makebox[1.0em]{#1})\hfil}
\newcommand{\Rn}[1]{%
	\textup{\lowercase\expandafter{\romannumeral#1}}%
}

\usepackage{multicol}

\newcommand\x{\times}

\usepackage{tikz,xcolor}
\definecolor{lime}{HTML}{A6CE39}
\definecolor{lightblue}{rgb}{0.0, 0.0, 0.5}
\DeclareRobustCommand{\orcidicon}{%
	\begin{tikzpicture}
	\draw[lime, fill=lime] (0,0)
	circle [radius=0.16]
	node[white] {{\fontfamily{qag}\selectfont \tiny ID}};
	\draw[white, fill=white] (-0.0625,0.095)
	circle [radius=0.007];
	\end{tikzpicture}
	\hspace{-2mm}
}
\foreach \x in {A, ..., Z}{%
	\expandafter\xdef\csname orcid\x\endcsname{\noexpand\href{https://orcid.org/\csname orcidauthor\x\endcsname}{\noexpand\orcidicon}}
}

\begin{document}

\title[
\textbf{RGDBEK}: \textbf{R}andomized \textbf{G}reedy \textbf{D}ouble \textbf{B}lock \textbf{E}xtended \textbf{K}aczmarz Algorithm with Hybrid Parallel Implementation and Applications]{\textbf{RGDBEK}: \textbf{R}andomized \textbf{G}reedy \textbf{D}ouble \textbf{B}lock \textbf{E}xtended \textbf{K}aczmarz Algorithm with Hybrid Parallel Implementation and Applications}

\author{\fnm{Aneesh} \sur{Panchal$^1$\orcidC}}\email{aneeshp@iisc.ac.in}

\author*[]{\fnm{Ratikanta} \sur{Behera$^2$\orcidA}}\email{ratikanta@iisc.ac.in}

\affil[]{\orgdiv{Department of Computational and Data Sciences}, \orgname{Indian Institute of Science}, \orgaddress{ \state{Bengaluru}, \country{India}}}


\vspace{.3cm}





\abstract{
Kaczmarz is one of the most prominent iterative solvers for linear systems of equations. 
Despite substantial research progress in recent years, the state-of-the-art Kaczmarz algorithms have not fully resolved the seesaw effect, a major impediment to convergence stability. Furthermore, while there have been advances in parallelizing the inherently sequential Kaczmarz method, no existing architecture effectively supports initialization-independent parallelism that fully leverages both CPU and GPU resources. This paper proposes the Randomized Greedy Double Block Extended Kaczmarz (RGDBEK) algorithm, a novel Kaczmarz approach designed for efficient large-scale linear system solutions. RGDBEK employs a randomized selection strategy for column and row blocks based on residual-derived probability distributions, thereby mitigating the traditional seesaw effect and enhancing convergence robustness. Theoretical analysis establishes linear convergence of the method under standard assumptions. Extensive numerical experiments on synthetic random matrices and real-world sparse matrices from the \texttt{SuiteSparse} collection demonstrate that RGDBEK outperforms existing Kaczmarz variants, including GRK, FDBK, FGBK, and GDBEK, in both iteration counts and computational time. In addition, a hybrid parallel CPU-GPU implementation utilizing optimized sparse matrix-vector multiplications via the state-of-the-art storage format improves scalability and performance on large sparse problems. Applications in finite element discretizations, image deblurring, and noisy population modeling demonstrate the algorithm's versatility and effectiveness. Future work will explore extending RGDBEK to tensor systems, optimizing parallel parameter selection, and reducing communication overhead to further enhance efficiency and applicability.
}

\maketitle

\section{Introduction}\label{sec:introduction}
Large-scale linear systems of equations of the form,
\begin{equation}\label{eq:Ax=b}
    Ax = b,\qquad \text{where, }A\in\mathbb{R}^{m\times n},\; b\in\mathbb{R}^{m\times 1},\text{ and } x\in\mathbb{R}^{n\times 1},
\end{equation}

arise in a wide range of scientific and engineering disciplines, including computed tomography~\cite{tomography-application}, image reconstruction~\cite{image-application}, signal processing~\cite{signal-processing-application, Xiao25}, neural network~\cite{tensor-application}, and inverse problems~\cite{inverseproblem-application, DongWang25}. Although direct solvers can theoretically recover the exact solution of $Ax=b$, they are typically infeasible in large-scale settings due to excessive fill-in, rapidly increasing memory requirements, and poor scalability. On the other hand, iterative solvers are well-suited to high-dimensional problems, as they rely primarily on matrix–vector products and vector updates. These operations are highly compatible with modern parallel and vector-processing hardware~\cite{wang2009solving}, making iterative algorithms the preferred choice for large systems~\cite{ahamed2012iterative}.

Among iterative techniques, the \textit{Kaczmarz algorithm}~\cite{kaczmarz1937} is one of the most classical and widely studied. This method iteratively projects the current approximation onto the hyperplanes defined by individual equations of the system. If the $i$-th row of $A$ is denoted by $a_i^\top$, the update at iteration $k$ can be written as,  
\begin{equation}
    x_{k+1} = x_k + \frac{b_i - \langle a_i, x_k \rangle}{\|a_i\|_2^2} a_i, \qquad i \equiv k \bmod m,
\end{equation}

where $\langle\cdot,\cdot\rangle$ represents the inner product.

This update scheme, also known as the \textit{Algebraic Reconstruction Technique (ART)}~\cite{gordon1970algebraic}, is popular due to its simplicity and low per-iteration computational cost. While the Kaczmarz method converges to the exact solution of consistent systems, its performance strongly depends on the order in which rows are processed. To address this limitation, Strohmer and Vershynin~\cite{strohmer2009randomized} proposed the \textit{Randomized Kaczmarz (RK) algorithm}, where rows are selected with probabilities proportional to the squared norms of the corresponding row vectors. Building on this idea, Bai and Wu~\cite{GRK} introduced the \textit{Greedy Randomized Kaczmarz (GRK)} method, which replaces random sampling with a greedy strategy that selects the row producing the maximum weighted residual reduction. Niu and Zheng~\cite{niu2020greedy} extended this principle to blocks of rows, leading to the \textit{Greedy Block Kaczmarz (GBK)} algorithm, which enhances parallelism by updating with multiple equations simultaneously. Building on greedy block selection strategies, Chen and Huang~\cite{FDBK} proposed the \textit{Fast Deterministic Block Kaczmarz (FDBK)} method, which performs greedy row selection without requiring pseudoinverse computations, thereby significantly reducing the computational cost. Following this, Xiao et al.~\cite{FGBK} introduced the \textit{Fast Greedy Block Kaczmarz (FGBK)} algorithm, which employs a more general greedy criterion by selecting the top $\eta\%$ of rows for a threshold parameter $\eta$ to construct greedy blocks. This approach not only accelerates computation but also opens avenues for further investigation into generalized greedy selection strategies. For inconsistent systems, Zouzias and Freris~\cite{zouzias2013randomized} developed the \textit{Randomized Extended Kaczmarz (REK)} method, which augments the Kaczmarz update with a residual-correction step. Its iterations are given by,  
\begin{equation}
    z_{k+1} = z_k - \frac{A_{j_k}^\top z_k}{\|A_{j_k}\|_2^2} A_{j_k}, 
    \qquad 
    x_{k+1} = x_k + \frac{\left(b^{i_k} - z^{i_k}_{k+1} - A^{i_k}x_k\right)}{\|A^{i_k}\|_2^2}\,(A^{i_k})^\top,
\end{equation}

where $z_k$ represents an auxiliary residual estimate and $\{i_k, j_k\}$ denote sampled row and column indices. Needell et al.~\cite{needell2015randomized} extended this idea to the \textit{Randomized Double Block Kaczmarz (RDBK)} method, which partitions both rows and columns into blocks to exploit greater parallelism. Subsequently, Du et al.~\cite{du2020randomized} proposed the \textit{Randomized Extended Block Kaczmarz (REBK)} algorithm, introducing pseudoinverse-based block updates that yield faster convergence guarantees. More recently, Bai and Wu~\cite{bai2021greedy} developed the \textit{Greedy Randomized Augmented Kaczmarz (GRAK)} method, which integrates greedy selection with the residual-corrected updates of REK. By selecting augmented blocks that maximize the reduction of the residual norm, GRAK achieves improved stability and accelerated convergence compared to previous approaches. Sun and Qin~\cite{GDBEK} recently proposed the \textit{Greedy Double Block Extended Kaczmarz (GDBEK)} algorithm, which adaptively forms both row and column index sets by employing a greedy strategy that selects subvectors corresponding to the largest residual norms. The iteration of GDBEK proceeds as,
\begin{equation}\label{eq:updateGDBEK}
    z_{k+1} = z_k - A_{\mathcal{U}_k} A_{\mathcal{U}_k}^\dagger z_k, \qquad x_{k+1} = x_k + (A^{\mathcal{J}_k})^\dagger \left(b^{\mathcal{J}_k} - z_{k+1}^{\mathcal{J}_k} - A^{\mathcal{J}_k} x_k^{\mathcal{J}_k}\right),
\end{equation}

where $z_k$ represents an auxiliary residual, $\mathcal{U}_k$ is the selected column index set and $\mathcal{J}_k$ is the selected row index set at iteration $k$.

However, this update rule iteratively selects the most dominant residuals, which may lead to a large number of iterations before convergence due to a seesaw effect, as observed by Li et al.~\cite{li2025global}. Hence, a more robust variant of GDBEK is required, which accounts not only for large residuals but also incorporates mid and low magnitude residuals to ensure improved convergence behavior.

Due to the implicit sequential nature of the Kaczmarz method, its efficient parallelization has long been a central topic of research. Efforts to parallelize Kaczmarz-type methods began with Martinez and Sampaio~\cite{martinez1986parallel}, who investigated independent Kaczmarz cycles on linearly independent row partitions under restrictive assumptions. Elble et al.~\cite{elble2010gpu} implemented Kaczmarz iterations on GPUs, demonstrating the potential of fine-grained parallel acceleration compared to other classical iterative methods. Liu et al.~\cite{liu2014asynchronous} proposed the \textit{Asynchronous Parallel Randomized Kaczmarz Algorithm (AsyncPRKA)} for shared-memory systems, although scalable, AsyncPRKA requires row preconditioning and therefore cannot directly handle inconsistent systems. Wang et al.~\cite{wang2022prkp} proposed the \textit{Distributed-Memory Randomized Kaczmarz Projection (PRKP)}, which integrates greedy row sampling with lazy updates to achieve performance on loosely synchronized clusters. More recently, Ferreira et al.~\cite{ferreira2024parallelization} investigated parallel Kaczmarz variants for dense systems, though their techniques do not exploit matrix sparsity. Bolukbacsi et al.~\cite{bolukbacsi2024distributed} introduced the \textit{Parallel Randomized Kaczmarz for sparse matrices (PARK)}, a distributed-memory method that, however, imposes explicit row normalization along with static row-block partitioning of $A$. It is evident that no existing parallel Kaczmarz algorithm simultaneously leverages both CPU and GPU architectures for computation. Moreover, most parallel Kaczmarz variants are tailored for dense matrices or depend on preconditioning techniques to achieve convergence.

To overcome these limitations, we propose the \textit{Randomized Greedy Double Block Extended Kaczmarz (RGDBEK)} algorithm, which incorporates the following key contributions,
\begin{enumerate}
    \item A row–column block selection strategy based on probability distributions derived from residuals. This strategy enhances convergence rates, reduces the required number of iterations, and reduces the seesaw effect, all without the need for explicit preconditioning.
    \item A scalable hybrid parallel implementation tailored for large-scale sparse matrices, effectively using both CPU and GPU architectures to maximize computational efficiency.
    \item A comprehensive experimental evaluation on real-world applications, including image deblurring, finite element methods, and signal recovery, thereby demonstrating the effectiveness and robustness of the proposed approach.
\end{enumerate}


The paper is organized as follows: Section \ref{sec:algorithm} explains the proposed RGDBEK algorithm in detail, Section \ref{sec:results} contains the numerical results and analysis of the algorithm, Section \ref{sec:PRGDBEK} explains the proposed parallel architecture for RGDBEK algorithm in detail, Section \ref{sec:ParallelResults} contains the results and analysis of the parallel architecture followed by some applications of the proposed RGDBEK algorithm in Section \ref{sec:applications}, and finally, Section \ref{sec:conclusions} concludes the paper with future directions.

\section{Proposed Algorithm}\label{sec:algorithm}
In this section, we are going to solve the linear system of equations given by Eq.~\eqref{eq:Ax=b}. Define auxiliary update vector $z \in \mathbb{R}^m$, which is initialized as $z_0 = b$. The update equations of the proposed algorithm follow those of the GDBEK method as specified in Eq.~\eqref{eq:updateGDBEK}.

Building on the framework of the GDBEK algorithm, the RGDBEK algorithm is presented in Algorithm~\ref{alg:rgdbek}. The GDBEK algorithm greedily selects the subsets of columns $\mathcal{U}_k$ and rows $\mathcal{J}_k$, according to the residuals defined by,  
\begin{equation*}
\begin{split}
        \mathcal{U}_k &= \left\{j_k:\frac{|A^\top_{(j_k)} z|^2}{\|A_{(j_k)}\|_2^2}\geq\eta\max_{1\leq j\leq n}\frac{|A^\top_{(j)} z|^2}{\|A_{(j)}\|_2^2}\right\}, \\
        \mathcal{J}_k &= \left\{i_k:\frac{|b^{(i_k)} - z^{(i_k)} - A^{(i_k)} x|^2}{\|A^{(i_k)}\|_2^2}\geq\eta\max_{1\leq i\leq m}\frac{|b^{(i)} - z^{(i)} - A^{(i)} x|^2}{\|A^{(i)}\|_2^2}\right\}.
\end{split}
\end{equation*}

This approach greedily selects indices according to the threshold parameter $\eta$, effectively ignoring residuals below this threshold. To address this limitation, we propose a novel randomized selection strategy for subsets of columns $\mathcal{U}_k$ and rows $\mathcal{J}_k$ based on probability distributions derived from the residuals, defined as follows,  
\begin{equation*}
    \text{Column Residuals: } \epsilon_{j_k}^z = \frac{|A^\top_{(j_k)} z_k|^2}{\|A_{(j_k)}\|_2^2}, \quad \text{Column Set Probabilities: } P(j_k) = \frac{\epsilon_{j_k}^z}{\sum_l\epsilon_{j_l}^z}.
\end{equation*}
\begin{equation*}
    \text{Row Residuals: } \epsilon_{i_k}^x = \frac{|b^{(i_k)} - z_{k+1}^{(i_k)} - A^{(i_k)} x_k|^2}{\|A^{(i_k)}\|_2^2},\quad \text{Row Set Probabilities: } P(i_k) = \frac{\epsilon_{i_k}^x}{\sum_l\epsilon_{i_l}^x}.
\end{equation*}

From these two equations, we select $n \eta$ number of columns according to the probabilities $P(j_k)$, forming the submatrix $A_{\mathcal{U}_k}$, and select $m \eta$ number of rows according to the probabilities $P(i_k)$, forming the submatrix $A^{\mathcal{J}_k}$. Since these rows and columns are selected randomly, RGDBEK allows updates from residuals of lower or medium magnitude as well. The theoretical convergence of RGDBEK, established in Theorem~\ref{thm:rgdbek_convergence}, demonstrates the linear convergence of the RGDBEK algorithm. However, in practical settings, RGDBEK often outperforms existing algorithms due to its ability to minimize the seesaw effect, as illustrated in the subsequent sections.

\begin{remark}
    In RGDBEK, the parameter, $\eta$ directly influences the convergence behavior. Increasing $\eta$ enlarges the size of the selected submatrices, which increases the per-iteration computational cost but lowers the iteration count. Conversely, a smaller $\eta$ reduces the per-iteration cost but requires more iterations to reach convergence.
\end{remark}

\begin{algorithm}
\caption{RGDBEK Algorithm}
\label{alg:rgdbek}
\begin{algorithmic}[1]
\State \textbf{Input:} $A \in \mathbb{R}^{m \times n}$, $b \in \mathbb{R}^m$, iterations $T$, initial guesses $x_0 = 0 \in \mathbb{R}^n, z_0 = b \in \mathbb{R}^m, $ parameter $\eta$
\State \textbf{Output:} $x_T$
\For{$k = 0,1,\dots,T-1$}
    \State Sample column block $\mathcal{U}_k$ with probabilities
    \State $\epsilon_{j_k}^z = \frac{|A^\top_{(j_k)} z_k|^2}{\|A_{(j_k)}\|_2^2}$
    \State $P(j_k) = \frac{\epsilon_{j_k}^z}{\sum_l\epsilon_{j_l}^z}$
    \State Select $n\eta$ columns using probability $P(j_k)$ which forms $A_{\mathcal{U}_k}$ 
    \State Update: $z_{k+1} = z_k - A_{\mathcal{U}_k} A_{\mathcal{U}_k}^\dagger z_k$ 
    \State Sample row block $\mathcal{J}_k$ with probabilities
    \State $\epsilon_{i_k}^x = \frac{|b^{(i_k)} - z_{k+1}^{(i_k)} - A^{(i_k)} x_k|^2}{\|A^{(i_k)}\|_2^2}$
    \State $P(i_k) = \frac{\epsilon_{i_k}^x}{\sum_l\epsilon_{i_l}^x}$
    \State Select $m\eta$ rows using probability $P(i_k)$ which forms $A^{\mathcal{J}_k}$
    \State Update: $x_{k+1} = x_k + (A^{\mathcal{J}_k})^\dagger \left(b^{\mathcal{J}_k} - z_{k+1}^{\mathcal{J}_k} - A^{\mathcal{J}_k} x_k^{\mathcal{J}_k}\right)$
\EndFor
\end{algorithmic}
\end{algorithm}

First, we establish the Block Sampling Lemma and the Geometric Inequality Lemma, which will serve as foundational results in the subsequent proof of convergence for the RGDBEK algorithm.
\begin{lemma}[Block sampling lemma]\label{lem:block_sampling}
Let $U \in \mathbb{R}^{d \times p}$, $v \in  \mathbb{R}^d$. Sample a block $\mathcal{B}$ of size $s = \eta d$ with probability
$$\mathbb{P}(\mathcal{B}) \propto \sum_{j \in \mathcal{B}} w_j^2, \quad w_j = \frac{|\langle u_j, v \rangle|}{\|u_j\|_2},$$
where $u_j^\top$ is the $j$th row of $U$. Then,
$$\mathbb{E} \|P_{\mathcal{B}} v\|^2_2 \geq \mu_{\text{block}}^{-1} \frac{\|U^\top v\|^2_2}{\|U\|_2^2},$$
where $\mu_{\text{block}} = \max_{\mathcal{B}'} \frac{\|U_{\mathcal{B}'}^\top U_{\mathcal{B}'}\|_2}{\|U_{\mathcal{B}'}\|_2^2}$, and $P_{\mathcal{B}} = U_{\mathcal{B}}^\top (U_{\mathcal{B}} U_{\mathcal{B}}^\top)^\dagger U_{\mathcal{B}}$ is projection matrix\footnote{Moore-Penrose pseudoinverse identity for matrix $X$: $X^\dagger = X^\top \left(XX^\top\right)^\dagger$}.
\end{lemma}

\begin{proof}
Define weights $w_j = \frac{|\langle u_j, v \rangle|}{\|u_j\|_2}$. The sampling probability is given as,
$$\mathbb{P}(\mathcal{B}) = \frac{\sum_{j \in \mathcal{B}} w_j^2}{S}, \quad \text{and, }\,\, S = \sum_{\mathcal{B}'} \sum_{j \in \mathcal{B}'} w_j^2 = \binom{d-1}{s-1} \sum_{j=1}^d w_j^2 = \binom{d-1}{s-1}T,$$
because, each index $j$ appears in $\binom{d-1}{s-1}$ blocks.

\noindent
The expectation is bounded as,
\begin{align*}
\mathbb{E} \|P_{\mathcal{B}} v\|^2_2 
&\geq \mu_{\text{block}}^{-1} \mathbb{E} \left[ \frac{\|U_{\mathcal{B}}^\top v\|^2_2}{\|U_{\mathcal{B}}\|_2^2} \right] \qquad\text{(by definition of $\mu_{\text{block}}^{-1}$ and spectral bound)}\\
&= \mu_{\text{block}}^{-1} \frac{1}{S} \sum_{\mathcal{B}'} \left( \sum_{j \in \mathcal{B}'} w_j^2 \right) \frac{\sum_{j \in \mathcal{B}'} |\langle u_j, v \rangle|^2}{\|U_{\mathcal{B}'}\|_2^2} \\
&= \mu_{\text{block}}^{-1} \frac{1}{S} \sum_{\mathcal{B}'} \left( \sum_{j \in \mathcal{B}'} w_j^2 \right) \frac{\sum_{j \in \mathcal{B}'} w_j^2 \|u_j\|_2^2}{\|U_{\mathcal{B}'}\|_2^2}.
\end{align*}

\noindent
Since, $\|U_{\mathcal{B}'}\|_2 \leq \|U\|_2$, we obtain the expectation as,
$$\mathbb{E}\|P_\mathcal{B} v\|_2^2 \;\geq\; 
\mu_{\text{block}}^{-1}\,\frac{1}{S\|U\|_2^2}\,
\sum_{\mathcal{B}'} \left(\sum_{j \in \mathcal{B}'} w_j^2\right)
\left(\sum_{j \in \mathcal{B}'} w_j^2 \|u_j\|_2^2\right).$$

\noindent
By the double-counting argument, each pair $(j, k)$ contributes the same number of times. Simplifying, this entire sum reduces to,
$$\sum_{\mathcal{B}'} \left(\sum_{j \in \mathcal{B}'} w_j^2\right)
\left(\sum_{j \in \mathcal{B}'} w_j^2 \|u_j\|_2^2\right)
= \sum_{\mathcal{B}'}\sum_{j\in\mathcal{B}'}\sum_{k\in\mathcal{B}'}w_j^2w_k^2\|u_k\|_2^2 = \binom{d-1}{s-1}\,T \sum_{j=1}^d w_j^2 \|u_j\|_2^2.$$

\noindent
Hence, substituting this value back in the expectation, we get,
$$\mathbb{E}\|P_\mathcal{B} v\|_2^2 
\;\geq\; \mu_{\text{block}}^{-1}\,\frac{\sum_{j=1}^d w_j^2 \|u_j\|_2^2}{\|U\|_2^2} \;=\; \mu_{\text{block}}^{-1}\,\frac{\sum_{j=1}^d \langle u_j, v\rangle^2}{\|U\|_2^2} \;=\; \mu_{\text{block}}^{-1}\,\frac{\|U^\top v\|_2^2}{\|U\|_2^2}.$$

\noindent
This completes the proof of the Block sampling lemma.
\end{proof}

\begin{lemma}[Geometric inequality]\label{lem:geometric-inequality}
For $\gamma \in (0,1)$ and $k \geq 0$,
$$k \gamma^k \leq \frac{2\gamma^{k/2}}{1-\sqrt{\gamma}}.$$
\end{lemma}
\begin{proof}
Let $t = \sqrt{\gamma} \in (0,1)$. Then $k t^{2k} \leq \sup_{k \geq 0} k t^k \cdot t^k$. The maximum of $f(t) = k t^k$ occurs at $t = e^{-1/k}$, with $f(t) \leq e^{-1}$. Thus:
$$k t^{2k} \leq e^{-1} t^k \leq \frac{t^k}{1-t} \cdot e^{-1}(1-t) \leq \frac{2 t^k}{1-t},$$
since $e^{-1}(1-t) \leq 2$ for $t \in (0,1)$. Substituting $t = \sqrt{\gamma}$, we get the desired results.
\end{proof}

We now present the main convergence theorem for the proposed RGDBEK algorithm, which guarantees its linear convergence.
\begin{theorem}[Convergence of RGDBEK]
\label{thm:rgdbek_convergence}
Let $A \in \mathbb{R}^{m \times n}$, $b \in \mathbb{R}^m$ with $\text{rank}(A) = \mathbbmtt{r}$, and let $x^\star = A^\dagger b$ be the minimum-norm solution. Define $\mathbf{r} = (\mathbf{I} - AA^\dagger)b$ as the optimal residual. Under sampling parameter $\eta \in (0,1)$ and assuming row blocks satisfy $\sigma_{\min}(A^{\mathcal{J}}) \geq \nu > 0$, Algorithm \ref{alg:rgdbek} satisfies,
\begin{enumerate}[label=\alph*)]
    \item Residual convergence, 
    $\mathbb{E} \|z_k - \mathbf{r}\|^2_2 \leq \gamma^k \|z_0 - \mathbf{r}\|^2_2$,
    \item Solution convergence,
    $\mathbb{E} \|x_k - x^\star\|^2_2 \leq \gamma^k \|x_0 - x^\star\|^2_2 + \frac{2\gamma^{k/2} \|z_0 - \mathbf{r}\|^2_2}{\nu^2 (1-\sqrt{\gamma})}$,
\end{enumerate}
where $\gamma = 1 - \mu^{-1}\frac{\sigma_{\min}^2(A)}{\|A\|_2^2}$.
\end{theorem}

\begin{proof}
Define error terms,
\begin{align*}
e_k^z &:= z_k - \mathbf{r} \quad \text{(residual error)}, \\
e_k^x &:= x_k - x^\star \quad \text{(solution error)}.
\end{align*}

\noindent
The column update projects $z_k$ onto $\text{range}(A_{\mathcal{U}_k}^\perp)$,
$$z_{k+1} = (\mathbf{I} - P_{\mathcal{U}_k}) z_k, \quad e_{k+1}^z = (\mathbf{I} - P_{\mathcal{U}_k}) e_k^z,$$
where $P_{\mathcal{U}_k} = A_{\mathcal{U}_k} A_{\mathcal{U}_k}^\dagger$. Since, $\mathrm{\textbf{r}}\perp \text{range}(A)$ and $\text{range}(A_{\mathcal{U}_k})\subset\text{range}(A)$ hence, $P_{\mathcal{U}_k}\mathrm{\textbf{r}} = 0$. 

\noindent
According to Pythagoras' theorem, the norm evolves as follows,
\begin{equation}\label{eq:res_norm_evolve}
    \|e_{k+1}^z\|^2_2 = \|e_k^z\|^2_2 - \|P_{\mathcal{U}_k} e_k^z\|^2_2.
\end{equation}

\noindent
At iteration $k$, select a column block $\mathcal{U}_k$ of size $ \eta n$ with probability,
$$\mathbb{P}(\mathcal{U}_k) \propto \sum_{j \in \mathcal{U}_k} \frac{|\langle A_{(j)}, e_k^z \rangle|^2}{\|A_{(j)}\|_2^2},$$
where $A_{(j)}$ is the $j^{th}$ column of $A$, and $e_k^z = z_k - \mathbf{r}$ is the residual error.

\noindent
By Lemma \ref{lem:block_sampling} with $U = A$, $v = e_k^z$,
\begin{equation}\label{eq:res_expected}
    \mathbb{E} \|P_{\mathcal{U}_k} e_k^z\|^2_2 \geq \mu_{\text{col}}^{-1} \frac{\|A^\top e_k^z\|^2_2}{\|A\|_2^2} \geq \mu_{\text{col}}^{-1} \frac{\sigma_{\min}^2(A)}{\|A\|_2^2} \|e_k^z\|^2_2,    
\end{equation}
because, $\|A^\top e_k^z\|^2_2 \geq \sigma_{\min}^2(A)\|e_k^z\|^2_2$ as $e_k^z \in \text{range}(A)$. In this expression, $\mu_{\text{col}} = \max_{\mathcal{U}} \frac{\|A_{\mathcal{U}}^\top A_{\mathcal{U}}\|_2}{\|A_{\mathcal{U}}\|_2^2}$. Hence, using Eq.~\eqref{eq:res_norm_evolve} and \eqref{eq:res_expected},
$$\mathbb{E}[\|e_{k+1}^z\|^2_2] \leq \left(1 - \mu_{\text{col}}^{-1} \frac{\sigma_{\min}^2(A)}{\|A\|_2^2}\right) \|e_k^z\|^2_2 = \gamma_{\text{col}} \|e_k^z\|^2_2,$$
where $\gamma_{\text{col}} = 1 -  \mu_{\text{col}}^{-1} \frac{\sigma_{\min}^2(A)}{\|A\|_2^2}$. By induction\footnote{Using $\mathbb{E} \|z_{k+1} - \mathbf{r}\|^2_2 = \mathbb{E}\left[\mathbb{E} \left[\|z_{k+1} - \mathbf{r}\|^2_2|z_k\right]\right]$ to prove the induction.}, we can easily show that,
\begin{equation}\label{eq:res_error}
    \mathbb{E} \|z_k - \mathbf{r}\|^2_2 \leq \gamma_{\text{col}}^k \|z_0 - \mathbf{r}\|^2_2.
\end{equation}

\noindent
Now, the row update is given as,
\begin{align*}
x_{k+1} &= x_k + (A^{\mathcal{J}_k})^\dagger (b^{\mathcal{J}_k} - z_{k+1}^{\mathcal{J}_k} - A^{\mathcal{J}_k} x_k), \\
e_{k+1}^x &= (\mathbf{I} - Q_{\mathcal{J}_k}) e_k^x + (A^{\mathcal{J}_k})^\dagger e_{k+1}^{z,\mathcal{J}_k},
\end{align*}
where $Q_{\mathcal{J}_k} = (A^{\mathcal{J}_k})^\dagger A^{\mathcal{J}_k}$. The orthogonality of subspaces gives,
\begin{equation}\label{eq:sol_norm_evolve}
    \|e_{k+1}^x\|^2_2 = \|(\mathbf{I} - Q_{\mathcal{J}_k}) e_k^x\|^2_2 + \|(A^{\mathcal{J}_k})^\dagger e_{k+1}^{z,\mathcal{J}_k}\|^2_2.
\end{equation}

\noindent
At iteration $k$, select a row block $\mathcal{J}_k$ of size $ \eta m$ with probability,
$$\mathbb{P}(\mathcal{J}_k) \propto \sum_{i \in \mathcal{J}_k} \frac{|\langle A^{(i)}, e_k^x \rangle|^2}{\|A^{(i)}\|_2^2},$$
where $A^{(i)}$ is the $i^{th}$ row of $A$, and $e_k^x = x_k - x^\star$ is the solution error.

\noindent
Taking conditional expectation given $\sigma-$algebra $\mathcal{F}_k = \sigma\left( \{ \mathcal{U}_i \}_{i=0}^{k-1}, \{ \mathcal{J}_i \}_{i=0}^{k-1}, \{ z_i \}_{i=0}^k, \{ x_i \}_{i=0}^k \right)$,
\begin{equation}\label{eq:sol_norm1}
\mathbb{E}[\|e_{k+1}^x\|^2_2 \mid \mathcal{F}_k] = \mathbb{E}[\|(\mathbf{I} - Q_{\mathcal{J}_k}) e_k^x\|^2_2 \mid \mathcal{F}_k] + \mathbb{E}[\|(A^{\mathcal{J}_k})^\dagger e_{k+1}^{z,\mathcal{J}_k}\|^2_2 \mid \mathcal{F}_k].
\end{equation}

\noindent
By Lemma \ref{lem:block_sampling} applied to $U = A$, $v = e_k^x$:
\begin{equation}\label{eq:sol_norm2}
    \mathbb{E}[\|Q_{\mathcal{J}_k} e_k^x\|^2_2 \mid \mathcal{F}_k] \geq \mu_{\text{row}}^{-1} \frac{\|A^\top e_k^x\|^2_2}{\|A\|_2^2} \geq \mu_{\text{row}}^{-1} \frac{\sigma_{\min}^2(A)}{\|A\|_2^2} \|e_k^x\|^2_2,
\end{equation}
where, $\mu_{\text{row}} = \max_{\mathcal{J}} \frac{\|A^{\mathcal{J}} (A^{\mathcal{J}})^\top\|_2}{\|A^{\mathcal{J}}\|_2^2}$. Hence, using Eq.~\eqref{eq:sol_norm1} and \eqref{eq:sol_norm2},
\begin{equation}\label{eq:sol_norm_part1}
    \mathbb{E}[\|(\mathbf{I} - Q_{\mathcal{J}_k}) e_k^x\|^2_2 \mid \mathcal{F}_k] \leq \gamma_{\text{row}} \|e_k^x\|^2_2,
\end{equation}
with $\gamma_{\text{row}} = 1 -  \mu_{\text{row}}^{-1} \frac{\sigma_{\min}^2(A)}{\|A\|_2^2}$.

\noindent
Now, assuming $\sigma_{\min}(A^{\mathcal{J}_k}) \geq \nu > 0$ for all $\mathcal{J}_k$:
$$\|(A^{\mathcal{J}_k})^\dagger e_{k+1}^{z,\mathcal{J}_k}\|^2_2 \leq \frac{\|e_{k+1}^{z,\mathcal{J}_k}\|^2_2}{\sigma_{\min}^2(A^{\mathcal{J}_k})} \leq \frac{\|e_{k+1}^z\|^2_2}{\nu^2}.$$
Since $\mathcal{J}_k$ is $\mathcal{F}_k$-measurable \footnote{Solution error $e_k^x = x_k - x^\star$ is $\mathcal{F}_k$-measurable because $x_k$ is computed from updates using $\{\mathcal{U}_i, \mathcal{J}_i, z_i, x_i \}_{i=0}^{k-1}$ and the sampling probability for $\mathcal{J}_k$ depends explicitly on $\mathcal{F}_k$-measurable quantities (including $e_k^x$) making it $\mathcal{F}_k$-measurable.}
and $e_{k+1}^z$ is independent of $\mathcal{J}_k$ given $\mathcal{F}_k$:
\begin{equation}\label{eq:sol_norm_part2}
    \mathbb{E}[\|(A^{\mathcal{J}_k})^\dagger e_{k+1}^{z,\mathcal{J}_k}\|^2_2 \mid \mathcal{F}_k] \leq \frac{\mathbb{E}[\|e_{k+1}^z\|^2_2 \mid \mathcal{F}_k]}{\nu^2} \leq \frac{\gamma_{\text{col}}^{k+1} \|z_0 - \mathbf{r}\|^2_2}{\nu^2}.
\end{equation}

\noindent 
Hence, using equations \eqref{eq:sol_norm_evolve}, \eqref{eq:sol_norm_part1} and \eqref{eq:sol_norm_part2} we get,
$$\mathbb{E}[\|e_{k+1}^x\|^2_2 \mid \mathcal{F}_k] \leq \gamma_{\text{row}} \|e_k^x\|^2_2 + \frac{\gamma_{\text{col}}^{k+1} \|z_0 - \mathbf{r}\|^2_2}{\nu^2}.$$

\noindent
Taking the total expectation,
$$\mathbb{E} \|e_{k+1}^x\|^2_2 \leq \gamma_{\text{row}} \mathbb{E} \|e_k^x\|^2_2 + \frac{\gamma_{\text{col}}^{k+1} \|z_0 - \mathbf{r}\|^2_2}{\nu^2}.$$

\noindent
Let $\gamma = \max(\gamma_{\text{row}}, \gamma_{\text{col}})$ i.e., $\gamma_{\text{row}} \leq \gamma$ and $\gamma_{\text{col}} \leq \gamma$, we get,
\begin{equation}\label{eq:solution}
\mathbb{E}[\|e_k^x\|^2_2] \leq \gamma_{\text{row}}^k \|e_0^x\|^2_2 + \frac{\|z_0 - \mathbf{r}\|^2_2}{\nu^2} \gamma_{\text{col}} \sum_{j=0}^{k-1} \gamma_{\text{row}}^{k-1-j} \gamma_{\text{col}}^j \leq \gamma^k \|x_0 - x^\star\|^2_2 + \frac{k \gamma^k \|z_0 - \mathbf{r}\|^2_2}{\nu^2}.
\end{equation}
According to the Lemma \ref{lem:geometric-inequality}, $k \gamma^k \leq \frac{2\gamma^{k/2}}{1-\sqrt{\gamma}}$ for $\gamma \in (0,1)$, using Eq.~\eqref{eq:solution} we get,
\begin{equation}\label{eq:sol_error}
    \mathbb{E}[\|x_k - x^\star\|^2_2] \leq \gamma^k \|x_0 - x^\star\|^2_2 + \frac{2\gamma^{k/2} \|z_0 - \mathbf{r}\|^2_2}{(1-\sqrt{\gamma}) \nu^2}.
\end{equation}
Equations \eqref{eq:res_error} and \eqref{eq:sol_error} complete the required proof.
\end{proof}

\section{Experimental Results}\label{sec:results}
Sequential experiments (including comparisons) were conducted on a machine with Intel(R) Core(TM) $i$9-12900K CPU with 3.20GHz base frequency (Turbo up to 5.20GHz), 32 GB RAM, and 16 cores using NVIDIA GeForce GTX 1660 SUPER GPU with 6 GB of GDDR6 memory. We used MATLAB R2025a to compile and run the codes.

\begin{remark}
    In contrast to the GDBEK algorithm, where CGLS is employed for solving the small least-squares subproblems, we use the LSQR method. This choice is motivated by the observation of Bjorck et al.~\cite{bjorck1998stability}, who noted that CGLS tends to be more efficient than LSQR only when the matrix is well-conditioned, since CGLS relies on normal equations to solve the problem.
\end{remark}

\subsection{Random Matrices}
In all experiments, the coefficient matrix $A$ is generated using the MATLAB built-in function \texttt{sprandn} with $99\%$ sparsity. The termination criterion is based on the Relative Square Error (RSE), defined as,
\begin{equation}
    \mathrm{RSE} = \frac{\|Ax - b\|_2^2}{\|b\|_2^2}.
\end{equation}

For the experiments, the true solution vector $x_{\text{true}}$ is initialized with MATLAB's \texttt{randn} function, and the corresponding right-hand side vector $b$ is computed as $b = A x_{\text{true}} + r$, where $r \in \mathrm{Range}(A^\top)$. The initial solution $x_0$ is set to the zero vector, and the initial residual estimate $z_0$ is initialized as $b$. For all algorithms, the RSE termination threshold is set to $10^{-6}$ and the maximum number of iterations is limited to $400000$. For FGBK, GDBEK, and RGDBEK, the parameter $\eta$ is fixed at $0.5$. Additionally, for FGBK, the parameter $p$ is chosen as $2$.

Table~\ref{tab:fatmatrices} presents the average results for fat matrices ($m < n$), while Table~\ref{tab:tallmatrices} shows the average results for tall matrices ($m > n$). The speedups reported in these tables are calculated relative to the best-performing existing algorithms. It is evident from these tables that the proposed RGDBEK algorithm outperforms existing methods, GRK, FDBK, FGBK, and GDBEK in terms of both CPU time and the number of iterations. This performance advantage is also illustrated in Fig.~\ref{fig:cputime} and Fig.~\ref{fig:cpuiter}. Notably, the RGDBEK algorithm achieves a maximum speedup of $3.90$ times, showing the superiority of the proposed algorithm.

\begin{table}[!ht]
    \centering
    \resizebox{\textwidth}{!}{%
    \begin{tabular}{c c c c c c c}
    \hline
        \multicolumn{2}{c}{\textbf{Matrix Size ($\mathbf{m\times n}$)}} & $\mathbf{500\times 8000}$ & $\mathbf{1000\times 8000}$ & $\mathbf{1500\times 8000}$ & $\mathbf{2000\times 8000}$ & $\mathbf{2500\times 8000}$ \\\hline
        \multirow{3}{*}{GRK~\cite{GRK}} & CPU Time (s) & $24.729417$ & $-$ & $-$ & $-$ & $-$ \\
                              & Iterations & $1347.9$ & $-$ & $-$ & $-$ & $-$  \\
                              & Speedups & $-$ & $-$ & $-$ & $-$ & $-$ \\\hline
        \multirow{3}{*}{FDBK~\cite{FDBK}} & CPU Time (s) & $1.277403$ & $3.563295$ & $7.679426$ & $14.426834$ & $24.431726$ \\
                              & Iterations & $46.6$ & $68.9$ & $94.4$ & $129.3$ & $172.6$ \\
                              & Speedups & $-$ & $-$ & $-$ & $-$ & $-$ \\\hline
        \multirow{3}{*}{FGBK~\cite{FGBK}} & CPU Time (s) & $0.874342$ & $2.414013$ & $5.028953$ & $9.832166$ & $15.812446$ \\
                              & Iterations & $34.6$ & $49.1$ & $66.9$ & $89.0$ & $119.0$ \\
                              & Speedups & $1.00\times$ & $1.00\times$ & $1.00\times$ & $1.00\times$ & $1.00\times$ \\\hline
        \multirow{3}{*}{GDBEK~\cite{GDBEK}} & CPU Time (s) & $1.134633$ & $2.782978$ & $6.288686$ & $13.103065$ & $24.260070$ \\
                              & Iterations & $34.3$ & $52.1$ & $77.7$ & $113.9$ & $155.8$ \\
                              & Speedups & $-$ & $-$ & $-$ & $-$ & $-$ \\\hline
        \multirow{3}{*}{RGDBEK} & CPU Time (s) & $0.360229$ & $0.941294$ & $2.434340$ & $4.563865$ & $6.927544$ \\
                              & Iterations & $12.0$ & $14.0$ & $17.1$ & $20.9$ & $25.5$ \\
                              & Speedups & $2.43\times$ & $2.56\times$ & $2.07\times$ & $2.15\times$ & $2.28\times$ \\\hline
    \end{tabular}}
    \caption{Summary of average results over $10$ runs for random fat matrices ($m<n$) with sparsity $99\%$ with termination condition: $\frac{\|Ax-b\|^2_2}{\|b\|^2_2} \leq 10^{-6}$. (Results shown only if $\text{Time}\leq 100\ sec$ )}
    \label{tab:fatmatrices}
%
\vspace{.5cm}
    \centering
    \resizebox{\textwidth}{!}{%
    \begin{tabular}{c c c c c c c}
    \hline
        \multicolumn{2}{c}{\textbf{Matrix Size ($\mathbf{m\times n}$)}} & $\mathbf{8000\times 500}$ & $\mathbf{8000\times 1000}$ & $\mathbf{8000\times 1500}$ & $\mathbf{8000\times 2000}$ & $\mathbf{8000\times 2500}$ \\\hline
        \multirow{3}{*}{GRK~\cite{GRK}} & CPU Time (s) & $12.982912$ & $64.979697$ & $-$ & $-$ & $-$ \\
                              & Iterations & $882.4$ & $2030.1$ & $-$ & $-$ & $-$  \\
                              & Speedups & $-$ & $-$ & $-$ & $-$ & $-$ \\\hline
        \multirow{3}{*}{FDBK~\cite{FDBK}} & CPU Time (s) & $0.538999$ & $2.016343$ & $4.633323$ & $14.818112$ & $44.701397$ \\
                              & Iterations & $48.1$ & $75.7$ & $99.0$ & $139.8$ & $188.6$ \\
                              & Speedups & $-$ & $-$ & $-$ & $-$ & $-$ \\\hline
        \multirow{3}{*}{FGBK~\cite{FGBK}} & CPU Time (s) & $0.453167$ & $1.369696$ & $3.392065$ & $10.620041$ & $32.757391$ \\
                              & Iterations & $44.1$ & $61.0$ & $84.1$ & $110.2$ & $146.9$ \\
                              & Speedups & $1.00\times$ & $1.00\times$ & $1.00\times$ & $1.00\times$ & $1.00\times$ \\\hline
        \multirow{3}{*}{GDBEK~\cite{GDBEK}} & CPU Time (s) & $0.765653$ & $2.083566$ & $5.303517$ & $15.360687$ & $41.847565$ \\
                              & Iterations & $35.1$ & $52.6$ & $79.3$ & $112.3$ & $156.6$ \\
                              & Speedups & $-$ & $-$ & $-$ & $-$ & $-$ \\\hline
        \multirow{3}{*}{RGDBEK} & CPU Time (s) & $0.401969$ & $0.965721$ & $2.304293$ & $4.766493$ & $8.396985$ \\
                              & Iterations & $11.8$ & $14.1$ & $17.0$ & $$20.6$$ & $24.9$ \\
                              & Speedups & $1.27\times$ & $1.42\times$ & $1.47\times$ & $2.23\times$ & $3.90\times$ \\\hline
    \end{tabular}}
    \caption{Summary of average results over $10$ runs for random tall matrices ($m>n$) with sparsity $99\%$ with termination condition: $\frac{\|Ax-b\|^2_2}{\|b\|^2_2} \leq 10^{-6}$. (Results shown only if $\text{Time}\leq 100\ sec$ )}
    \label{tab:tallmatrices}
\end{table}

\begin{figure}[!ht]
    \centering
    \begin{subfigure}[b]{0.49\textwidth}
        \includegraphics[width=\textwidth]{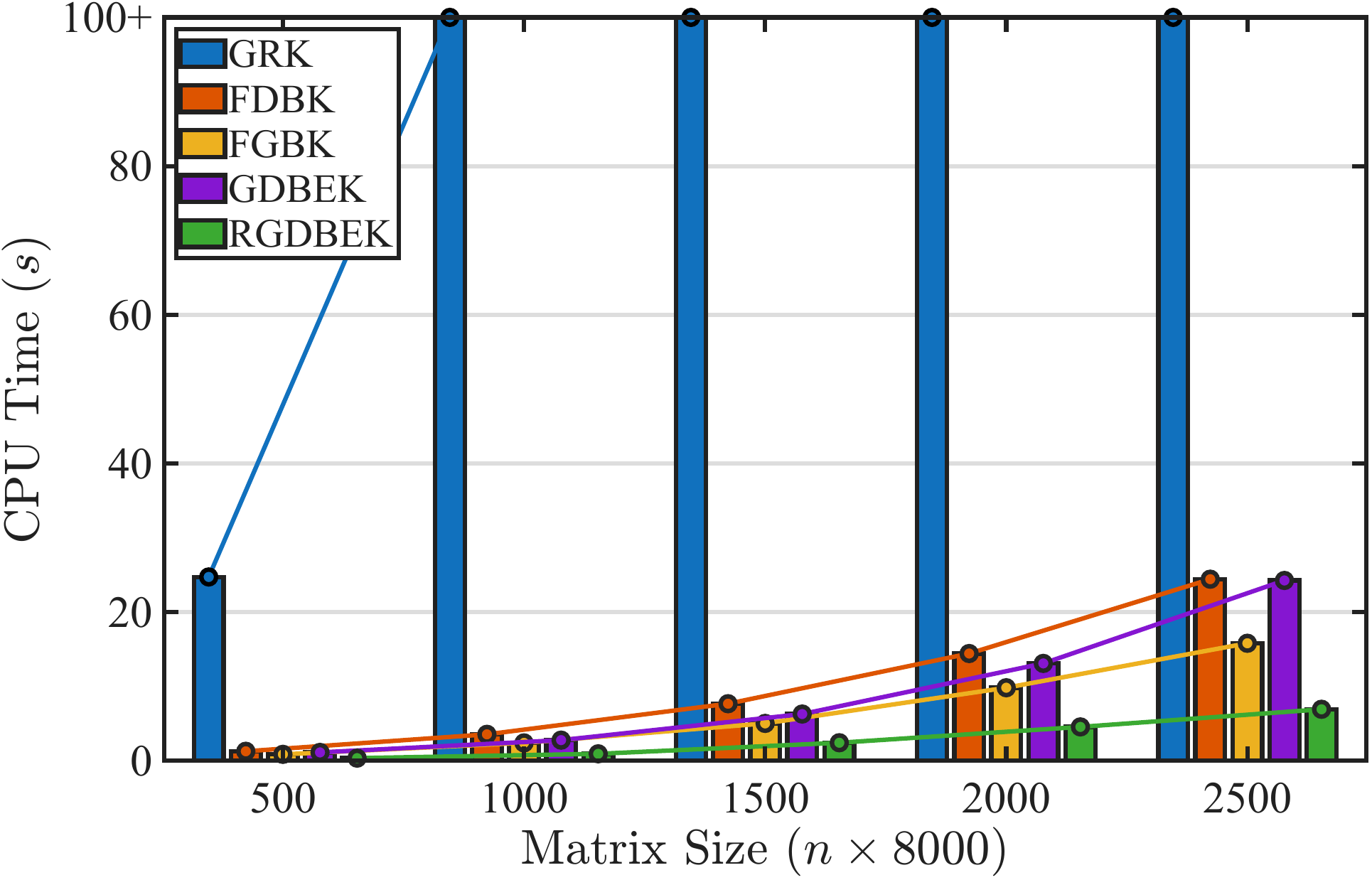}
        \caption{Sparse random fat matrices}
    \end{subfigure}
    \hfill
    \begin{subfigure}[b]{0.49\textwidth}
        \includegraphics[width=\textwidth]{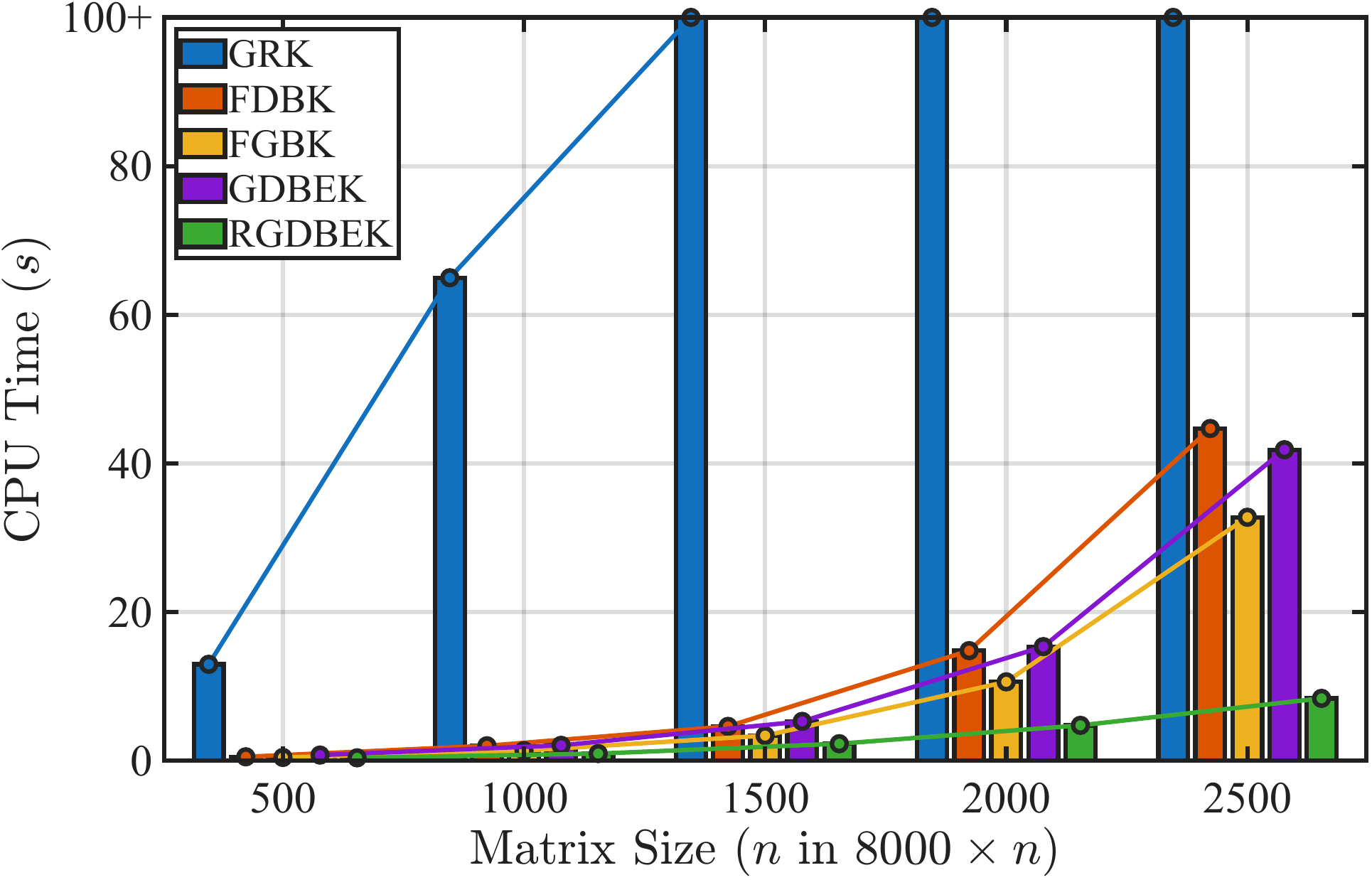}
        \caption{Sparse random tall matrices}
    \end{subfigure}
    \caption{CPU Time comparison using results given in Table~\ref{tab:fatmatrices} and Table~\ref{tab:tallmatrices}.}
    \label{fig:cputime}
%
\vspace{.75cm}
    \begin{subfigure}[b]{0.49\textwidth}
        \includegraphics[width=\textwidth]{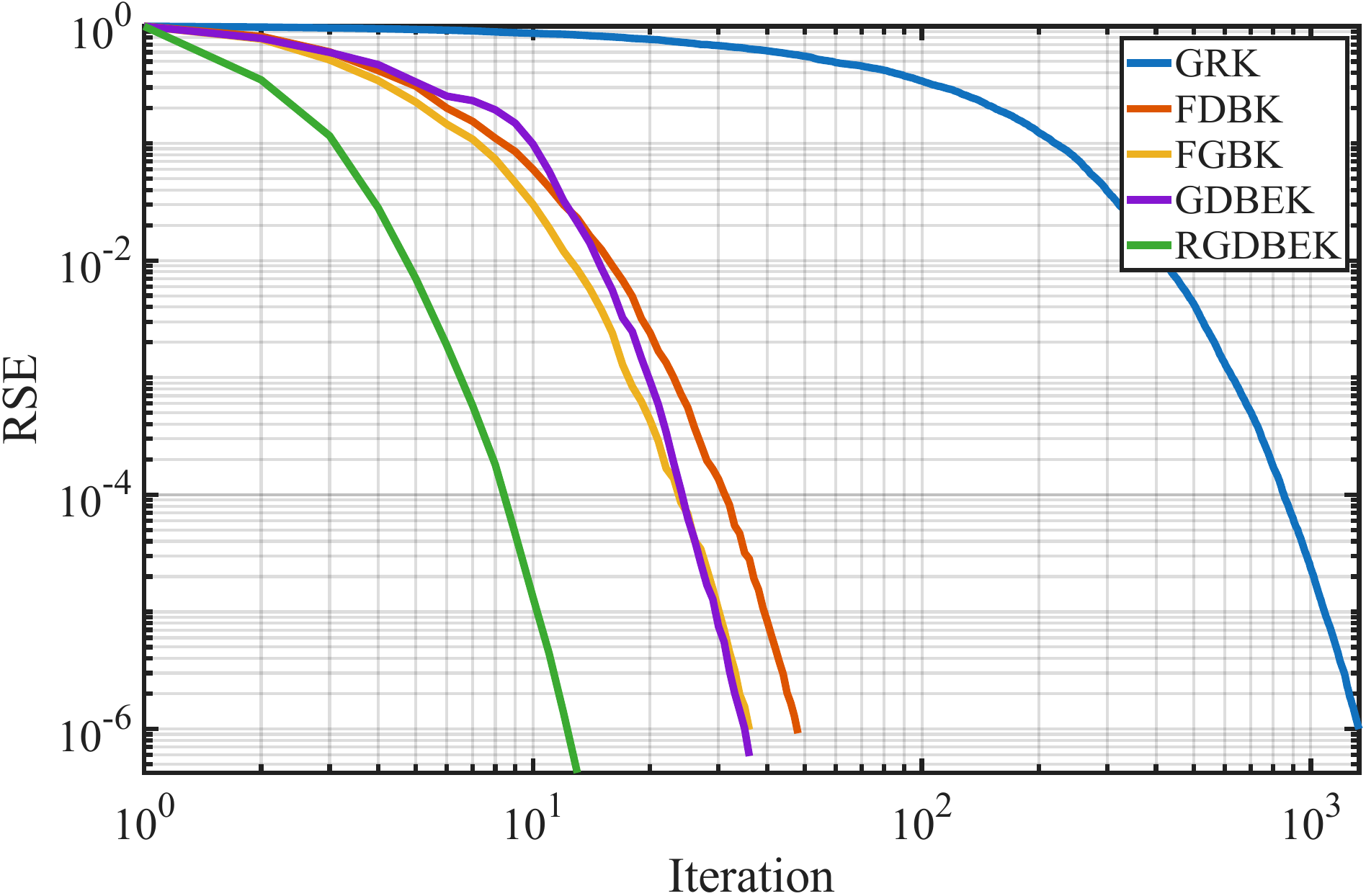}
        \caption{Relative Residual vs Iteration for $m=500$, $n=8000$}
    \end{subfigure}
    \hfill
    \begin{subfigure}[b]{0.49\textwidth}
        \includegraphics[width=\textwidth]{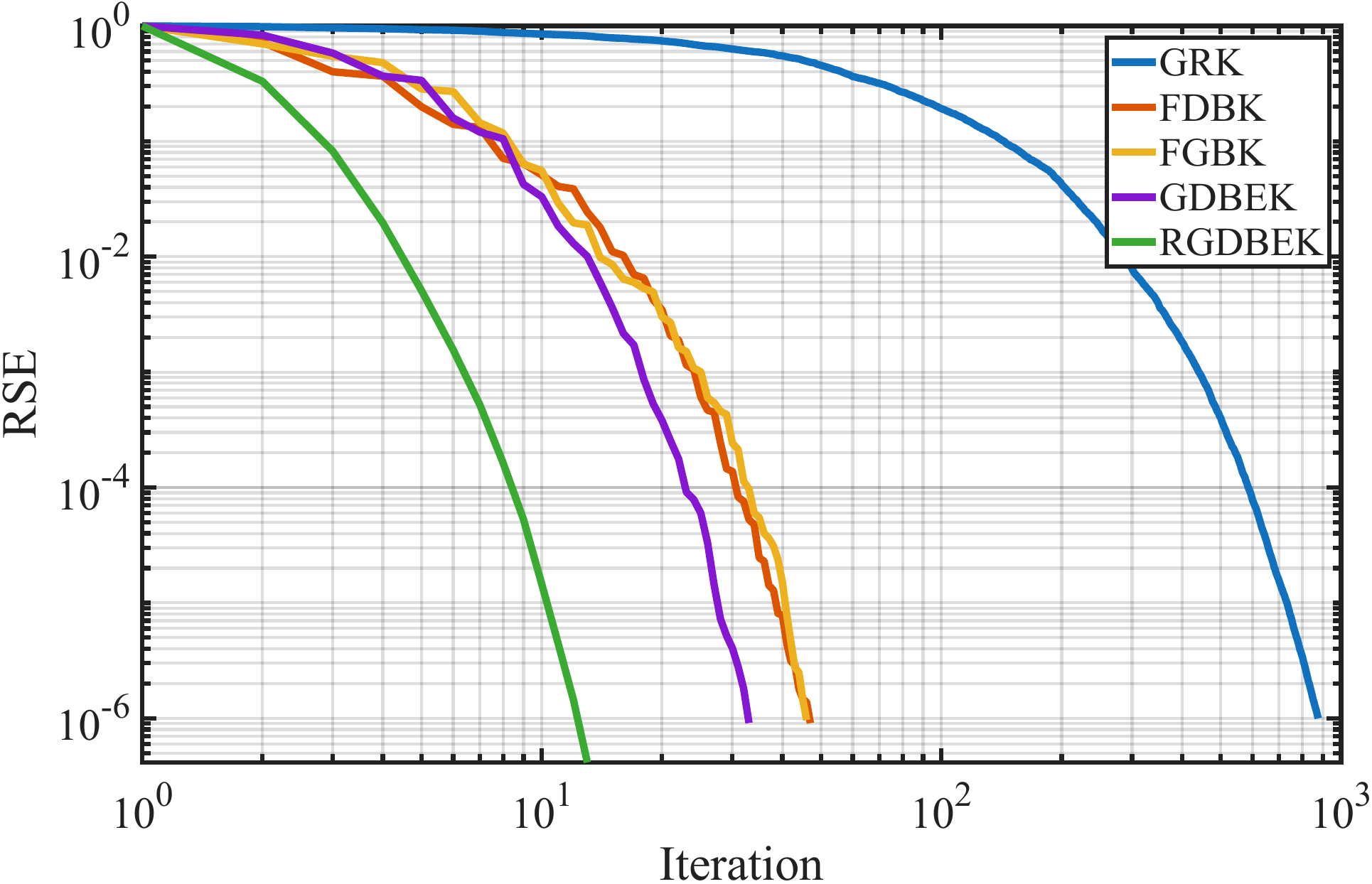}
        \caption{Relative Residual vs Iteration for $m=8000$, $n=500$}
    \end{subfigure}
    \caption{Convergence over iterations analysis using results given in Table~\ref{tab:fatmatrices} and Table~\ref{tab:tallmatrices}.}
    \label{fig:cpuiter}
\end{figure}

\subsection{SuiteSparse Matrix Collection}
In addition to experiments on randomly generated matrices, we conducted comparisons on matrices from the \texttt{SuiteSparse} matrix collection~\cite{suitesparse} with various sizes and sparsity patterns. The parameter settings for these experiments were kept consistent with those used in the random matrix tests. Furthermore, the vectors $x_{\text{true}}$, $b$, $x_0$, and $z_0$ were initialized following the same procedures as previously described.

Table~\ref{tab:suitesparse} presents the average results for various matrices from the \texttt{SuiteSparse} matrix collection~\cite{suitesparse}. The speedups reported in this table are calculated relative to the best-performing existing algorithms. It can be observed that the proposed RGDBEK algorithm outperforms the existing methods GRK, FDBK, FGBK, and GDBEK in terms of both CPU time and the number of iterations. This performance advantage is further illustrated in Fig.~\ref{fig:suitesparse}. Notably, the RGDBEK algorithm achieves a maximum speedup of $2.53$ times, demonstrating the superiority of the proposed approach. The only exception is observed for the \texttt{WorldCities} matrix, which shows a reduced speedup of $0.34$ times, potentially due to its small size, limiting the observable acceleration.

\begin{table}[!ht]
    \centering
    \resizebox{\textwidth}{!}{%
    \begin{tabular}{c c c c c c c}
    \hline
        \multicolumn{2}{c}{\textbf{Matrix}} & \textbf{WorldCities} & \textbf{lp\_agg2} & \textbf{GL7d26} & \textbf{GD96\_a} & \textbf{GL7d12} \\\hline
        \multicolumn{2}{c}{\textbf{Size ($\mathbf{m\times n}$)}} & $315\times 100$ & $516\times 758$ & $305\times 2798$ & $1096\times 1096$ & $8899\times 1019$ \\
        \multicolumn{2}{c}{\textbf{Rank}} & $100$ & $516$ & $273$ & $827$ & $960$ \\
        \multicolumn{2}{c}{\textbf{Sparsity}} & $76.13\%$ & $98.79\%$ & $99.13\%$ & $99.86\%$ & $99.59\%$ \\
        \multicolumn{2}{c}{\textbf{Condition Number}} & $65.999$ & $589.751$ & $\infty$ & $\infty$ & $3.792\times 10^{17}$ \\\hline
        \multirow{3}{*}{GRK~\cite{GRK}} & CPU Time (s) & $0.276238$ & $0.935025$ & $0.949980$ & $4.522051$ & $73.667474$ \\
                              & Iterations & $3239.6$ & $1458.0$ & $670.9$ & $2265.2$ & $2447.1$ \\
                              & Speedups & $-$ & $-$ & $-$ & $-$ & $-$ \\\hline
        \multirow{3}{*}{FDBK~\cite{FDBK}} & CPU Time (s) & $0.164150$ & $0.160193$ & $0.290375$ & $0.647073$ & $4.144814$ \\
                              & Iterations & $1694.6$ & $236.2$ & $138.6$ & $368.8$ & $150.2$ \\
                              & Speedups & $-$ & $-$ & $-$ & $-$ & $-$ \\\hline
        \multirow{3}{*}{FGBK~\cite{FGBK}} & CPU Time (s) & $0.138817$ & $0.132544$ & $0.252340$ & $0.531691$ & $3.290251$ \\
                              & Iterations & $1719.8$ & $213.3$ & $120.2$ & $313.1$ & $142.0$ \\
                              & Speedups & $1.00\times$ & $-$ & $-$ & $1.00\times$ & $-$ \\\hline
        \multirow{3}{*}{GDBEK~\cite{GDBEK}} & CPU Time (s) & $0.263940$ & $0.097652$ & $0.128780$ & $0.563247$ & $3.181373$ \\
                              & Iterations & $846.9$ & $85.6$ & $36.1$ & $197.9$ & $72.6$ \\
                              & Speedups & $-$ & $1.00\times$ & $1.00\times$ & $-$ & $1.00\times$ \\\hline
        \multirow{3}{*}{RGDBEK} & CPU Time (s) & $0.407082$ & $0.090252$ & $0.070001$ & $0.384606$ & $1.256719$ \\
                              & Iterations & $561.8$ & $32.1$ & $10.9$ & $67.2$ & $15.0$ \\
                              & Speedups & $0.34\times$ & $1.08\times$ & $1.84\times$ & $1.38\times$ & $2.53\times$ \\\hline
    \end{tabular}}
    \caption{Summary of average results over $10$ runs for matrices from \texttt{SuiteSparse} matrix collection~\cite{suitesparse} with termination condition: $\frac{\|Ax-b\|^2_2}{\|b\|^2_2} \leq 10^{-6}$.}
    \label{tab:suitesparse}
\end{table}

\begin{figure}[!ht]
    \centering
    \begin{subfigure}[b]{0.49\textwidth}
        \includegraphics[width=\textwidth]{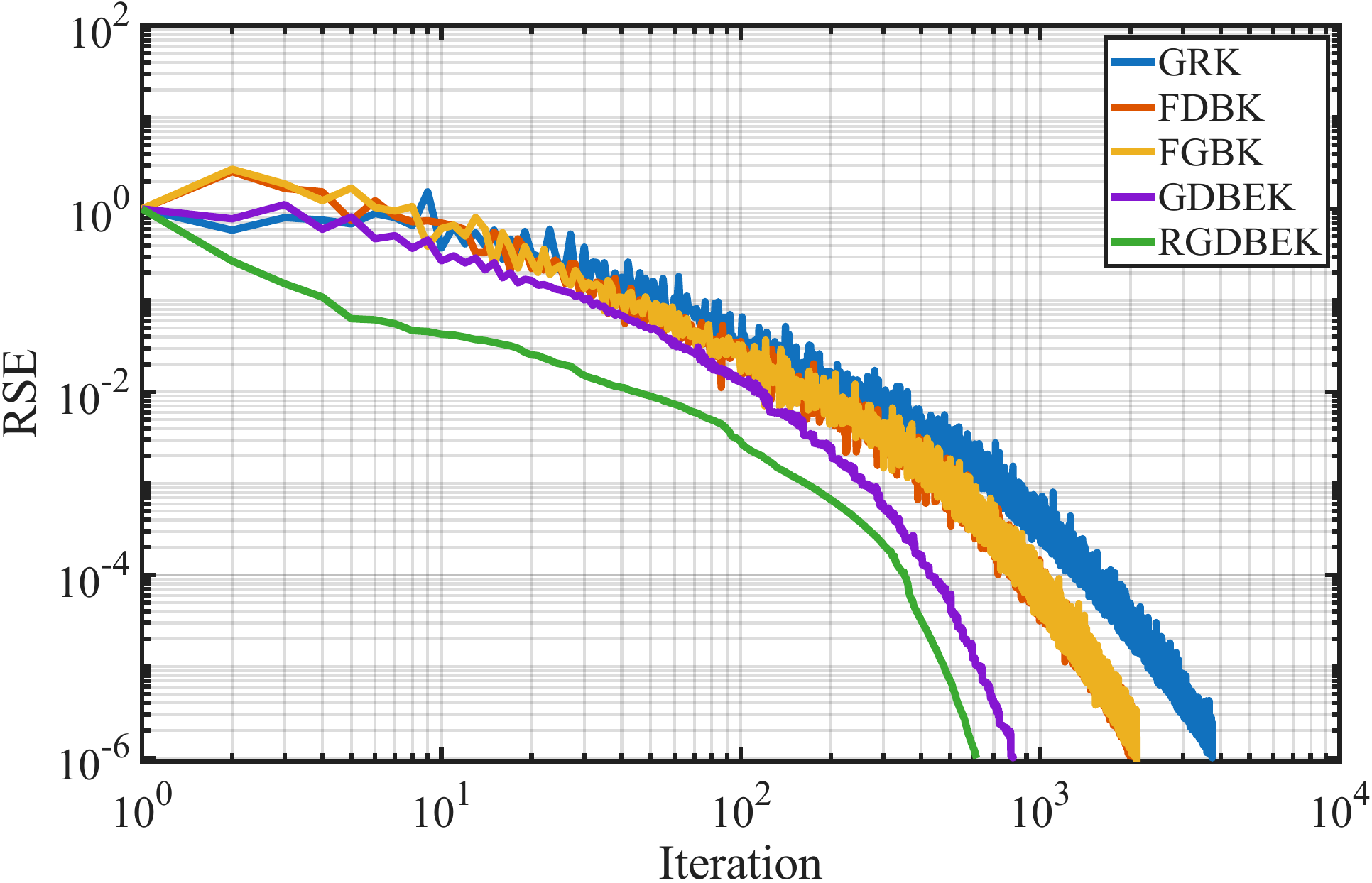}
        \caption{\texttt{WorldCities.mat} matrix of size $315\times 100$}
    \end{subfigure}
    \hfill
    \begin{subfigure}[b]{0.49\textwidth}
        \includegraphics[width=\textwidth]{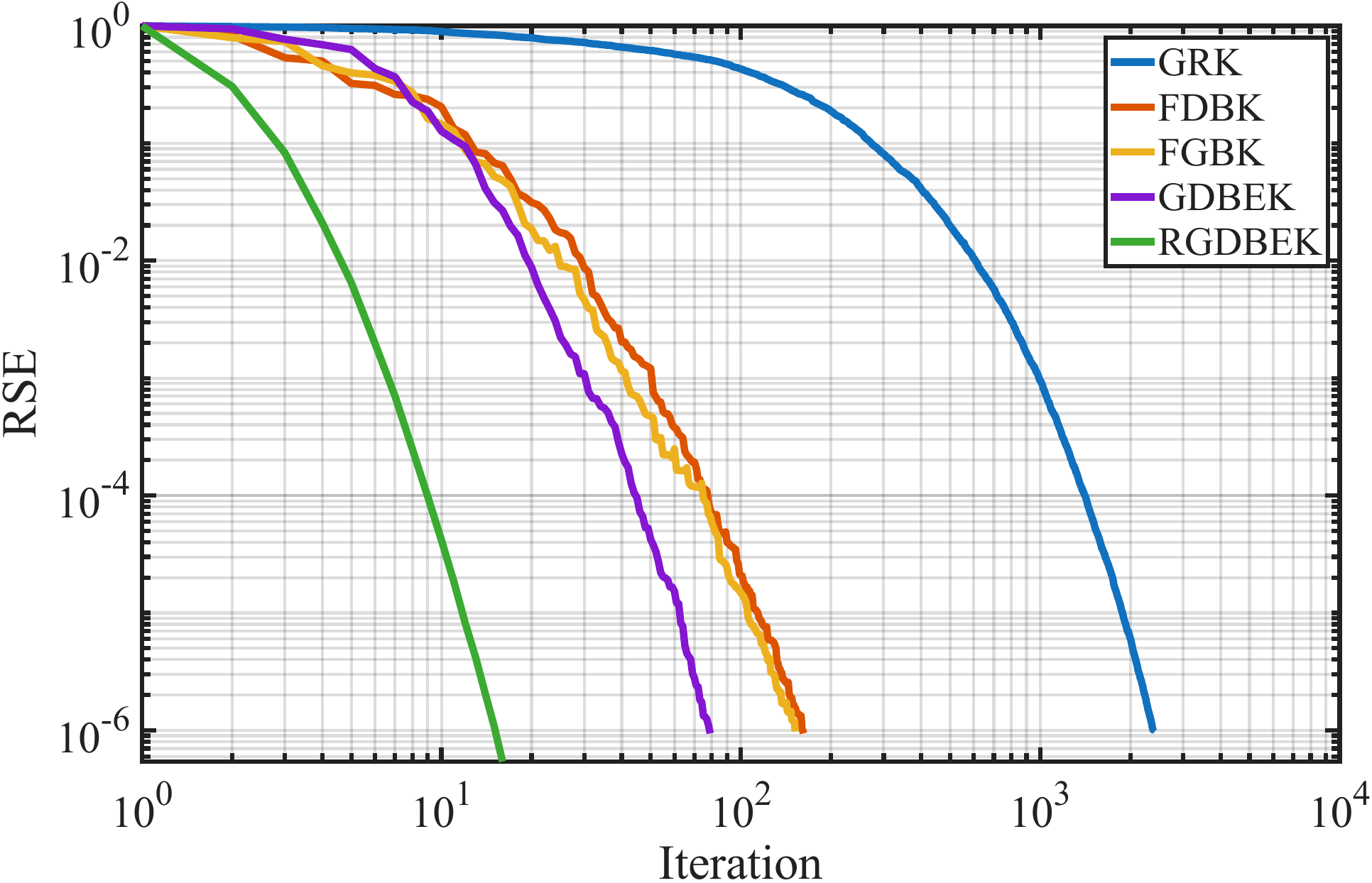}
        \caption{\texttt{GL7d12.mat} matrix of size $8899\times 1019$}
    \end{subfigure}
    \caption{Convergence over iterations analysis using results given in Table~\ref{tab:suitesparse}.}
    \label{fig:suitesparse}
\end{figure}

\section{Parallel Architecture for RGDBEK}\label{sec:PRGDBEK}
In practical scenarios, most matrices are inherently sparse, making sparse computations significantly more important than dense computations. It is important to note that most iterative methods spend the majority of their computational time performing sparse matrix-vector multiplication (SpMV) operations because iterative solvers often require hundreds or thousands of iterations to converge. Hence, optimizing SpMV kernels is critical to enhancing the overall performance of sparse solvers. To address this challenge, we adopt the \textit{Block Strategy and Adaptive Storage (BSAS)} format recently proposed by Zhao et al.~\cite{zhao2025block}, which integrates multiple existing formats, including COO, ELL, CSR, and hybrid schemes based on adaptive threshold criteria. This approach ensures workload balance, minimizes redundant calculations, and significantly enhances memory utilization efficiency. Hence, all SpMV operations are performed on the GPU using the BSAS storage format.

The parallel RGDBEK algorithm, presented in Algorithm~\ref{alg:rgdbek_bsas}, solves a large sparse linear system $Ax = b$ distributed across $P$ processes (MPI ranks), leveraging GPU-accelerated SpMV. The global matrix $A$ of size $m \times n$ is partitioned row-wise among the processors, with each rank assigned a contiguous block of rows until the number of non-zero entries is almost equally divided. The right-hand side vector $b$ is broadcast to all processors to ensure global accessibility. 

\begin{remark}
    All the Sparse Matrix Vector multiplications (SpMV) performed on GPU are placed inside \boxed{~} in the Algorithm \ref{alg:rgdbek_bsas}. 
\end{remark}

At each iteration, the parallel RGDBEK algorithm advances through two primary phases to iteratively refine the solution vector $x$. Initially, each processor computes its local contribution $(A^\top z)^{(p)} = (A^{(p)})^\top z_k^{(p)}$, where $z_k^{(p)}$ denotes the local auxiliary residual vector on rank $p$. These local contributions are aggregated across all processors using an \texttt{MPI\_AllReduce} operation to form the global vector $A^\top z_k$. A probability distribution over columns is then constructed, where each column $j$ is weighted by the squared magnitude $|(A^\top z)_j|^2$ normalized by the squared Euclidean norm of the column $\|A_{(:,j)}\|_2^2$. From this distribution, a subset $\mathcal{U}_k$ containing $n\eta$ columns is sampled. Each processor extracts the corresponding local submatrix $A^{(p)}_{\mathcal{U}_k}$ and solves a small local least squares problem to obtain a correction vector $z_{\mathrm{sol}}$. The auxiliary vector update is performed locally as $z_{k+1}^{(p)} = z_k^{(p)} - A^{(p)}_{\mathcal{U}_k} z_{\mathrm{sol}}$. In the second phase, the residual vector on each processor is computed as $r^{(p)} = b^{(p)} - z_{k+1}^{(p)} - A^{(p)} x_k$. A row-wise probability distribution is constructed where each row $i$ is weighted by $|r_i^{(p)}|^2$ normalized by the row norm $\|A^{(p)}_{(i,:)}\|_2^2$. Each processor samples $d\eta$ rows $\mathcal{J}_k$ from this distribution, extracts the corresponding row submatrix $A^{(p)}_{\mathcal{J}_k}$ and subvector $b^{(p)}_{\mathcal{J}_k}$, and solves a local least squares correction problem for $x_{\mathrm{update}}^{(p)}$. The global solution vector is updated collectively using an \texttt{MPI\_AllReduce} with a lazy averaging step. After each iteration, the root processor evaluates the RSE. If the RSE is below the prescribed tolerance, a termination signal is broadcast to all ranks to break the computation. Finally, the algorithm returns the global solution vector $x_T$ after convergence or reaching the maximum iterations. This process involves two principal synchronization points via \texttt{MPI\_AllReduce} calls for data aggregation and solution update, as illustrated in Fig.~\ref{fig:mpiflow}.

Efficient data transfer management between the CPU and GPU is critical for optimizing the performance of SpMV within the proposed RGDBEK algorithm. During each iteration $k$, intermediate vectors updated on the GPU, such as the residuals $r^{(p)}$, the auxiliary vectors $z_k^{(p)}$, and the partial solutions $x_k^{(p)}$, are kept on the GPU as much as possible to avoid frequent CPU-GPU data transfers. Only essential data, such as the sampled column submatrix solutions $z_{\mathrm{sol}}$ associated with $\mathcal{U}_k$ and the row submatrix updates $x_{\mathrm{update}}^{(p)}$ related to $\mathcal{J}_k$, are transferred back to the CPU for interprocess communication and global synchronization using \texttt{MPI\_AllReduce}. This strategy minimizes communication bottlenecks and supports overlapping computation with data transfers, thus promoting convergence efficiency in large-scale sparse linear systems.

\begin{algorithm}
\caption{Parallel version of RGDBEK Algorithm $\boxed{\texttt{GPU SpMV}}$}
\label{alg:rgdbek_bsas}
\begin{algorithmic}[1]
\State \textbf{Input:} $A \in \mathbb{R}^{m \times n}$, $b \in \mathbb{R}^m$, iterations $T$, initial guess $x_0 = 0 \in \mathbb{R}^n$, $z_0 = b \in \mathbb{R}^m$, parameter $\eta$, MPI distributed environment with $P$ processors
\State \textbf{Output:} $x_T$
\State \textbf{On rank 0:} Partition $A$ rows for all ranks as $A^{(p)}$, $p = 0 \ldots P-1$ and send each $A^{(p)}$ to rank $p$
\State \textbf{On all ranks:} Receive $A^{(p)} \in \mathbb{R}^{d \times n}$, where $d = m/p$
\State \textbf{Broadcast} $b\in\mathbb{R}^m$ to all ranks (\texttt{MPI\_Bcast})
\For{$k = 0, 1, \dots, T-1$}
    \State Each processor computes: $\left(A^\top z\right)^{(p)} = \boxed{(A^{(p)})^{\top} z_k^{(p)}}$
    \State \texttt{MPI\_AllReduce}: $A^\top z = \sum_{p=1}^P \left(A^\top z\right)^{(p)}$
    \State Compute:
    $\epsilon_j^z = \frac{\left|\left(A^\top z\right)_j\right|^2}{\left\|A_{(:,j)}\right\|_2^2}, \quad
      P_j = \frac{\epsilon_j^z}{\sum_j \epsilon_j^z}$
    \State Sample $n \eta$ columns using $P$ which forms set $\mathcal{U}_k$
    \State Each processor extracts submatrix $A^{(p)}_{\mathcal{U}_k}$ locally
    \State Solve local least squares problem using LSQR:
      $$z_{\mathrm{sol}} = \arg\min_y \left\|A^{(p)}_{\mathcal{U}_k} y - z_k^{(p)} \right\|_2$$
    \State Update $z$ locally:
      $$z_{k+1}^{(p)} = z_k^{(p)} - \boxed{A^{(p)}_{\mathcal{U}_k} z_{\mathrm{sol}}}$$

    \State Compute residual locally: $r^{(p)} = b^{(p)} - z_{k+1}^{(p)} - \boxed{A^{(p)} x_k}$
    \State Compute: $\epsilon_i^x = \frac{\left|r_i^{(p)}\right|^2}{\left\|A^{(p)}_{(i,:)}\right\|_2^2}, \quad P_i = \frac{\epsilon_i^x}{\sum_i \epsilon_i^x}$
    \State Sample $d \eta$ rows $\mathcal{J}_k$ according to $P$
    \State Extract row submatrix $A^{(p)}_{\mathcal{J}_k}$ and vector $b^{(p)}_{\mathcal{J}_k}$ locally
    \State Solve local least squares problem using LSQR:
      $$x^{(p)}_{\mathrm{update}} = \arg\min_y \left\| A^{(p)}_{\mathcal{J}_k} y - \left(b^{(p)}_{\mathcal{J}_k} - z_{k+1}^{(p)} - \boxed{A^{(p)}_{\mathcal{J}_k} x_k} \right) \right\|_2$$
    \State Update Solution using Lazy Approximation using MPI (\texttt{MPI\_AllReduce}):
      $$x_{k+1} = x_k + \frac{1}{P}\sum_{p=1}^Px^{(p)}_{\mathrm{update}}$$

    \If{Root Processor}
      \State Compute RSE: 
      $$\frac{\| b - \boxed{A x_{k+1}} \|_2^2}{\| b \|_2^2}$$
      \If{$RSE < \mathrm{tolerance}$}
        Broadcast \texttt{Terminate}
      \EndIf
    \EndIf
    \If{Terminate}
      Break
    \EndIf
\EndFor
\State \textbf{Return} $x_T$
\end{algorithmic}
\end{algorithm}

\begin{figure}
    \centering
    \includegraphics[width=\linewidth]{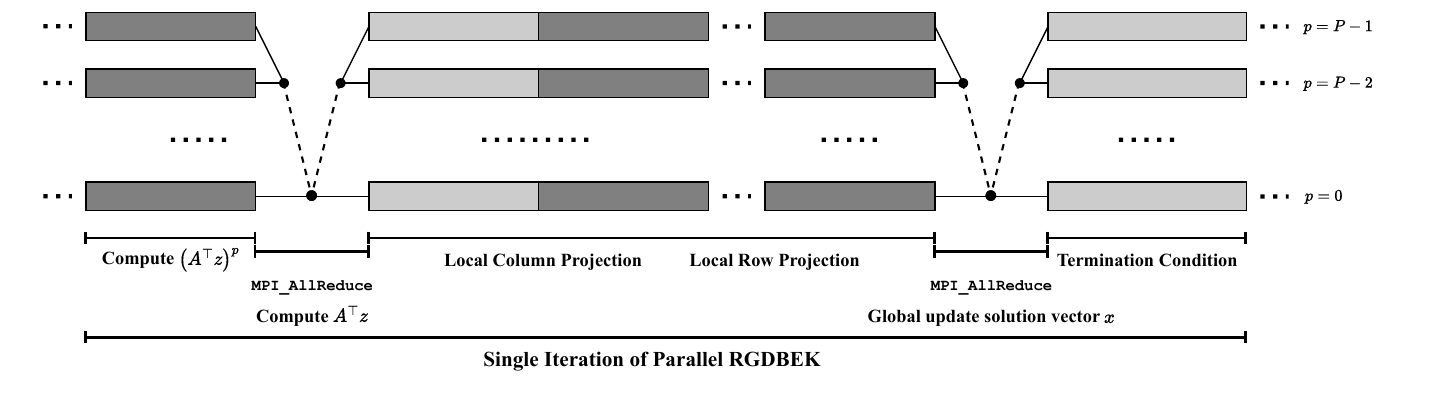}
    \caption{MPI Execution architecture for Parallel version of RGDBEK}
    \label{fig:mpiflow}
\end{figure}

\section{Parallel Experimental Results}\label{sec:ParallelResults}
Parallel architecture experiments (including comparisons) were conducted on a machine with Intel(R) Core(TM) $i$9-14900K CPU with 3.20GHz base frequency (Turbo up to 6.00GHz), 32 GB RAM, and 32 cores using NVIDIA RTX 4500 ADA Generation, 24570 GPU with 24 GB of GDDR6 memory. All parallel codes are created in Python 3.10.13 and executed using CUDA 12.4.131 and MPI 4.0.2. Python libraries used are \texttt{mpi4py} (3.1.3), \texttt{numpy} (1.26.4), \texttt{scipy} (1.15.3), and \texttt{pycuda} (2022.2.2).

In all experiments, the coefficient matrix $A$ is generated using the \texttt{scipy.sparse.rand} function with $99\%$ sparsity. The true solution vector $x_{\text{true}}$ is initialized with \texttt{numpy.random.rand} function. For all algorithms, the RSE termination threshold is set to $10^{-4}$ and the maximum number of iterations is limited to $500$. All the remaining settings are the same as sequential experiments.

Table~\ref{tab:parSpeedupsm} presents the average results for random matrices with a varying number of rows $m$. We observe notable speedups as the number of processes increases from $1$ to $4$ for all matrices. However, when the process count increases from $4$ to $8$, performance gains stagnate or degrade. This behavior can be attributed to the two mandatory global synchronization steps, one during the column projection step and another during the update of the global solution vector $\mathbf{x}$. These synchronization points create a bottleneck that limits scalability, and according to Algorithm~\ref{alg:rgdbek}, both steps are essential and cannot be removed. Despite this, the observed speedups up to $4$ processes (see Fig.~\ref{fig:parm:speedup}) surpass the ideal linear speedups of $2.00$ times and $4.00$ times. This superlinear speedup is due to lazy approximations partially addressing the seesaw effect during the global solution update, reducing the number of iterations required for convergence. Fig.~\ref{fig:parm:it} illustrates that increasing the number of processes decreases the iterations to reach the target threshold, thereby reducing total execution time beyond the expected linear scaling. This synergy between parallelism and algorithmic convergence leads to the observed high-performance improvements.

\begin{table}[!ht]
    \centering
    \resizebox{\textwidth}{!}{%
    \begin{tabular}{c c c c c c c}
    \hline
        \multicolumn{2}{c}{\textbf{Matrix Size ($\mathbf{m\times n}$)}} & $\mathbf{10000\times 20000}$ & $\mathbf{20000\times 20000}$ & $\mathbf{30000\times 20000}$ & $\mathbf{40000\times 20000}$ & $\mathbf{50000\times 20000}$ \\\hline
        \multirow{3}{*}{$1$ Process} & Iterations & $112.3$ & $383.5$ & $277.6$ & $205.7$ & $179.6$ \\
                              & Time (s) & $9.5339$ & $54.9513$ & $58.2495$ & $54.8381$ & $64.9093$ \\
                              & Speedups & $1.000\times$ & $1.000\times$ & $1.000\times$ & $1.000\times$ & $1.000\times$ \\\hline
        \multirow{3}{*}{$2$ Processes} & Iterations & $52.3$ & $102.7$ & $122.2$ & $144.7$ & $109.5$ \\
                              & Time (s) & $3.5570$ & $9.5441$ & $14.5025$ & $20.6103$ & $18.5490$ \\
                              & Speedups & $2.680\times$ & $5.758\times$ & $4.016\times$ & $2.661\times$ & $3.499\times$ \\\hline
        \multirow{3}{*}{$4$ Processes} & Iterations & $48.5$ & $76.3$ & $79.3$ & $75.0$ & $80.7$ \\
                              & Time (s) & $3.7646$ & $7.1169$ & $8.5717$ & $9.2071$ & $11.0228$ \\
                              & Speedups & $2.532\times$ & $7.721\times$ & $6.796\times$ & $5.956\times$ & $5.889\times$ \\\hline
        \multirow{3}{*}{$8$ Processes} & Iterations & $94.9$ & $76.5$ & $66.3$ & $65.3$ & $59.7$ \\
                              & Time (s) & $15.3341$ & $9.8432$ & $9.0941$ & $9.7216$ & $9.5194$ \\
                              & Speedups & $0.622\times$ & $5.558\times$ & $6.405\times$ & $5.641\times$ & $6.819\times$ \\\hline
    \end{tabular}}
    \caption{Summary of average results over $10$ runs for random matrices (change in $m$) with $99\%$ sparsity with termination condition: $\mathrm{RSE} \leq 10^{-4}$, and $\eta = 0.1$.}
    \label{tab:parSpeedupsm}
\end{table}

\begin{figure}[!ht]
    \centering
    \begin{subfigure}[b]{0.49\textwidth}
        \includegraphics[width=\textwidth]{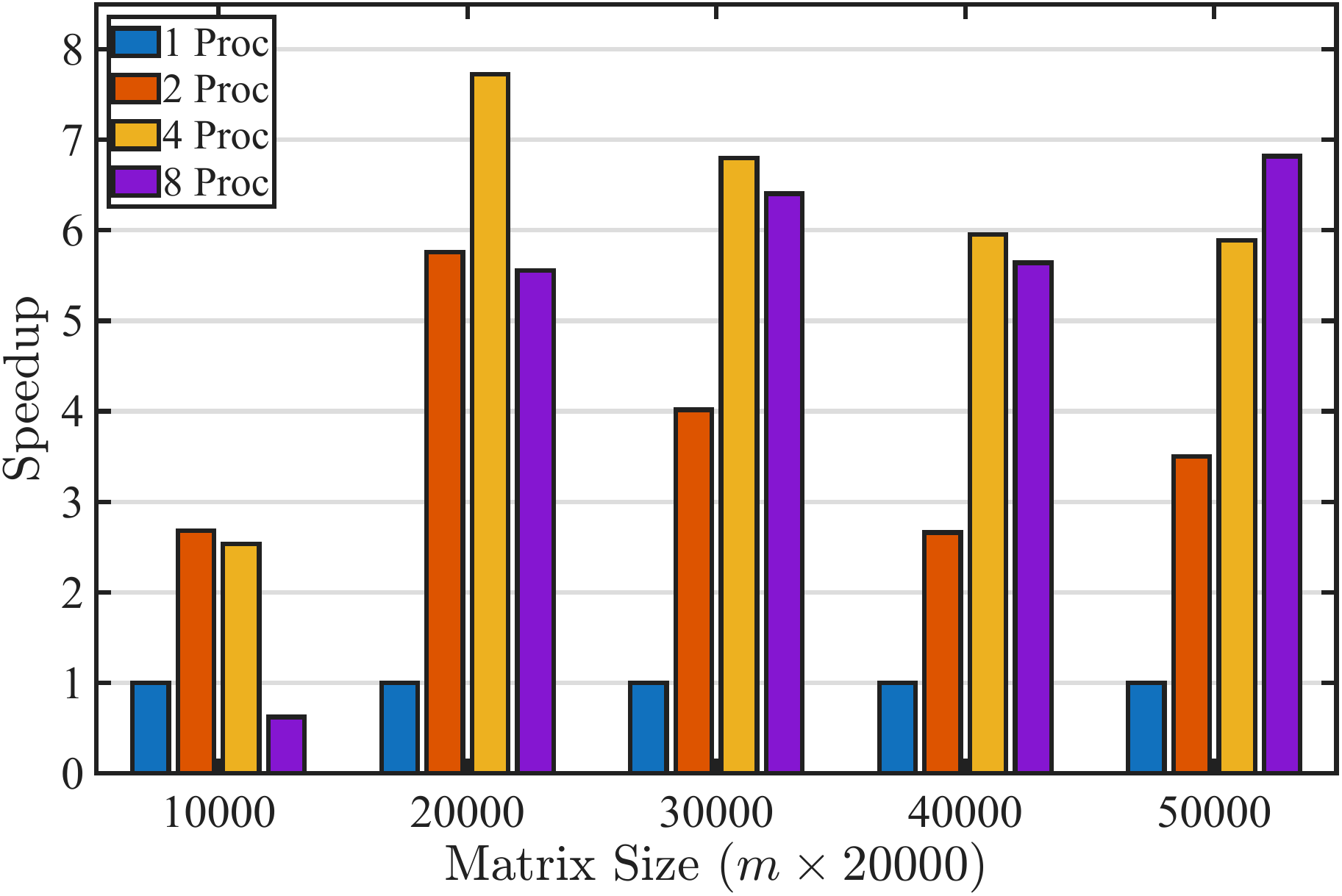}
        \caption{Speedups over different numbers of processes}
        \label{fig:parm:speedup}
    \end{subfigure}
    \hfill
    \begin{subfigure}[b]{0.49\textwidth}
        \includegraphics[width=\textwidth]{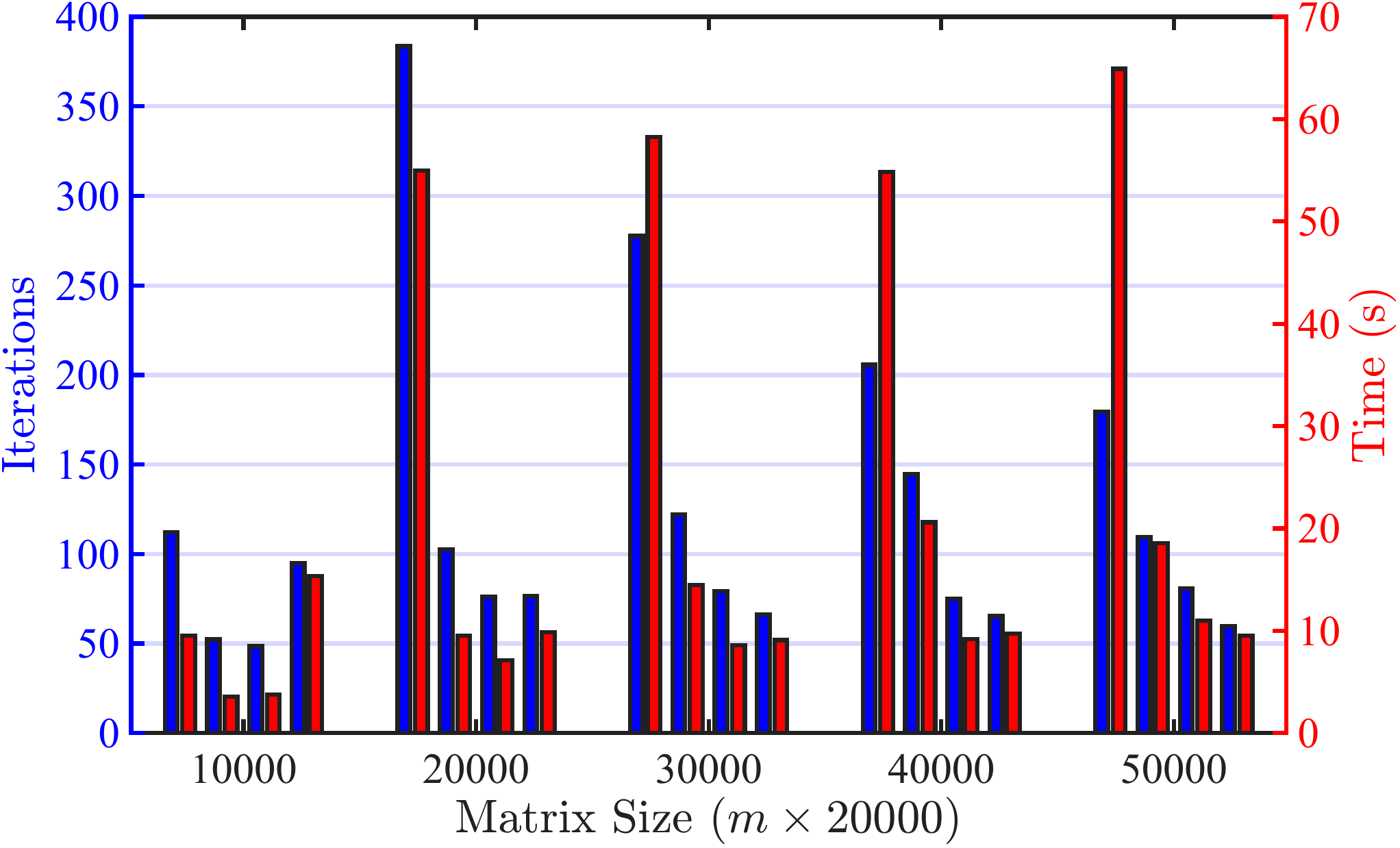}
        \caption{Iterations and Time dependence analysis}
        \label{fig:parm:it}
    \end{subfigure}
    \caption{Time, Iterations, and Speedups comparisons using results given in Table~\ref{tab:parSpeedupsm}.}
    \label{fig:parm}
\end{figure}

On similar grounds, Table~\ref{tab:parSpeedupsn} presents the average results for random matrices with a varying number of columns $n$. Consistent with the previous observations, we see significant speedups as the process count increases from $1$ to $4$, followed by stagnation or degraded performance when scaling from $4$ to $8$ processes. Fig.~\ref{fig:parn:speedup} and \ref{fig:parn:it} illustrate the speedup trends and the time-iteration relationship, respectively, for different matrix sizes. The explanation for the superlinear speedups up to $4$ processes and the subsequent stagnation is the same reasoning described earlier.

\begin{table}[!ht]
    \centering
    \resizebox{\textwidth}{!}{%
    \begin{tabular}{c c c c c c c}
    \hline
        \multicolumn{2}{c}{\textbf{Matrix Size ($\mathbf{m\times n}$)}} & $\mathbf{20000\times 10000}$ & $\mathbf{20000\times 20000}$ & $\mathbf{20000\times 30000}$ & $\mathbf{20000\times 40000}$ & $\mathbf{20000\times 50000}$ \\\hline
        \multirow{3}{*}{$1$ Process} & Iterations & $157.3$ & $383.5$ & $197.8$ & $131.2$ & $101.4$ \\
                              & Time (s) & $14.0463$ & $54.9512$ & $43.1847$ & $37.2935$ & $37.8090$ \\
                              & Speedups & $1.000\times$ & $1.000\times$ & $1.000\times$ & $1.000\times$ & $1.000\times$ \\\hline
        \multirow{3}{*}{$2$ Processes} & Iterations & $166.8$ & $102.7$ & $80.8$ & $64.7$ & $56.0$ \\
                              & Time (s) & $10.0689$ & $9.5441$ & $10.1979$ & $10.4816$ & $11.4067$ \\
                              & Speedups & $1.395\times$ & $5.758\times$ & $4.235\times$ & $3.558\times$ & $3.315\times$ \\\hline
        \multirow{3}{*}{$4$ Processes} & Iterations & $68.8$ & $76.3$ & $63.3$ & $56.2$ & $45.2$ \\
                              & Time (s) & $4.7086$ & $7.1169$ & $7.3365$ & $8.4748$ & $8.6456$ \\
                              & Speedups & $2.983\times$ & $7.721\times$ & $5.886\times$ & $4.401\times$ & $4.373\times$ \\\hline
        \multirow{3}{*}{$8$ Processes} & Iterations & $58.7$ & $76.5$ & $69.0$ & $61.1$ & $55.1$ \\
                              & Time (s) & $9.2192$ & $9.8432$ & $8.6867$ & $9.0392$ & $9.5734$ \\
                              & Speedups & $1.524\times$ & $5.558\times$ & $4.971\times$ & $4.126\times$ & $3.949\times$ \\\hline
    \end{tabular}}
    \caption{Summary of average results over $10$ runs for random matrices (change in $n$) with $99\%$ sparsity with termination condition: $\mathrm{RSE} \leq 10^{-4}$, and $\eta = 0.1$.}
    \label{tab:parSpeedupsn}
\end{table}

\begin{figure}[!ht]
    \centering
    \begin{subfigure}[b]{0.49\textwidth}
        \includegraphics[width=\textwidth]{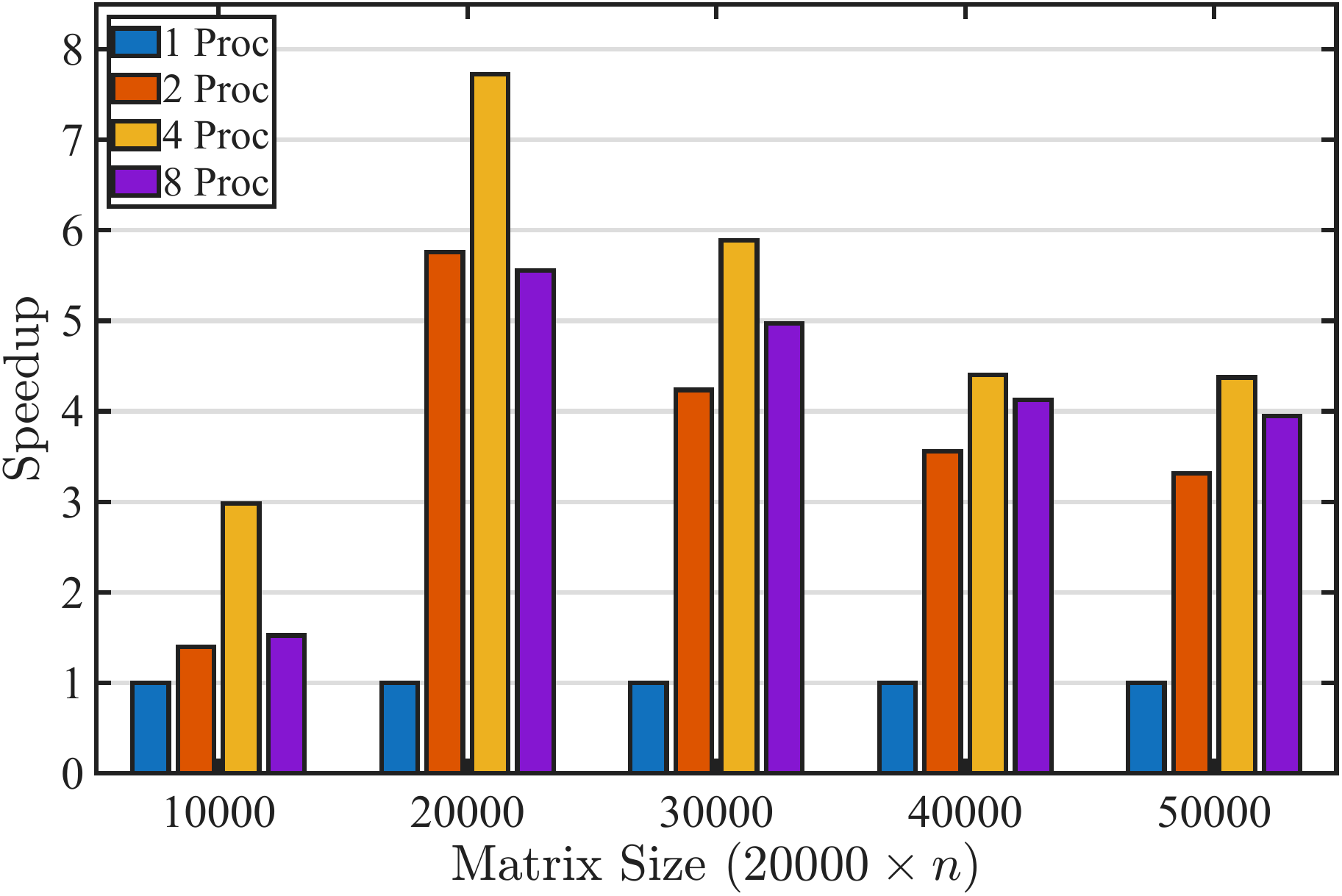}
        \caption{Speedups over different numbers of processes}
        \label{fig:parn:speedup}
    \end{subfigure}
    \hfill
    \begin{subfigure}[b]{0.49\textwidth}
        \includegraphics[width=\textwidth]{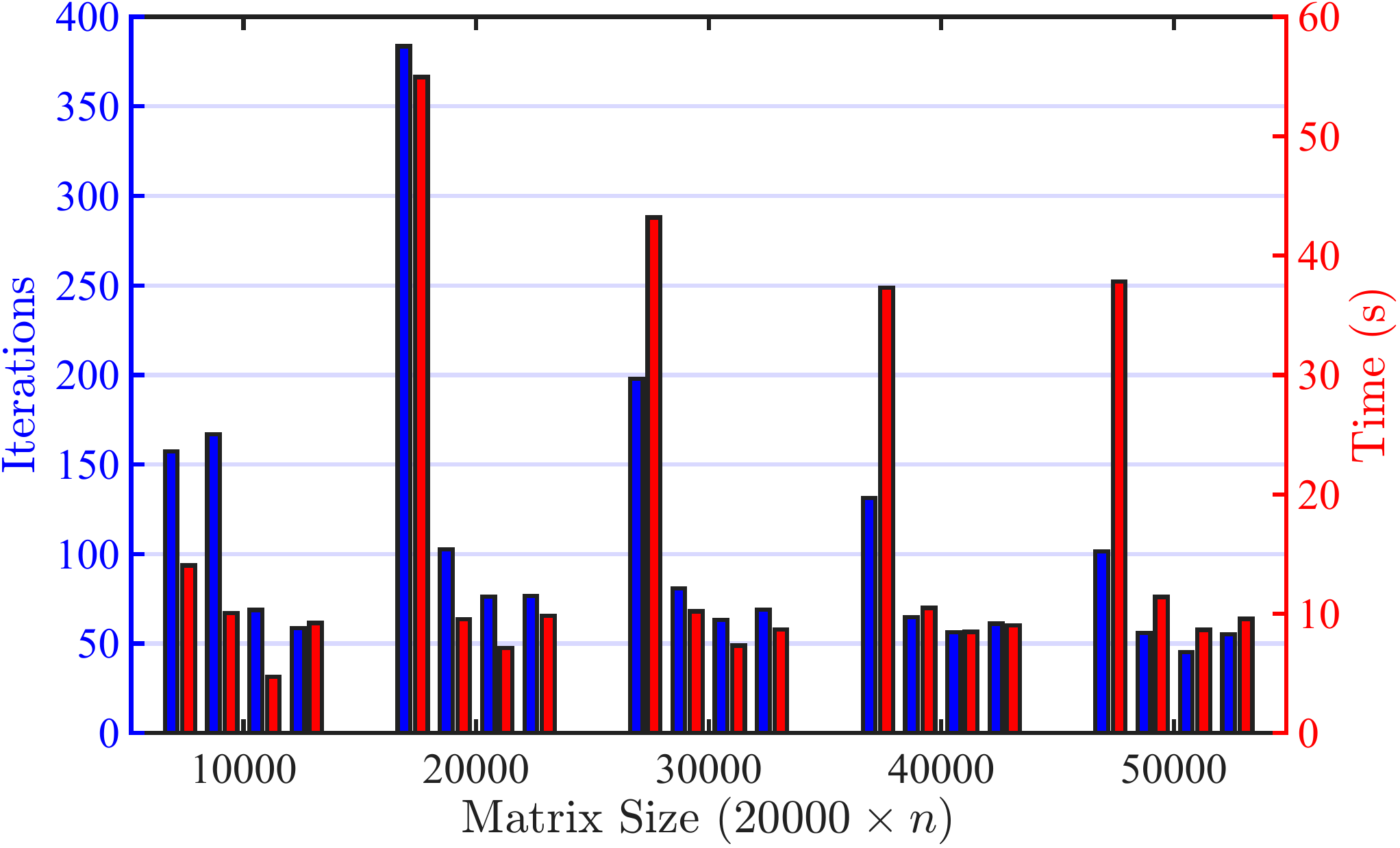}
        \caption{Iterations and Time dependence analysis}
        \label{fig:parn:it}
    \end{subfigure}
    \caption{Time, Iterations, and Speedups comparisons using results given in Table~\ref{tab:parSpeedupsn}.}
    \label{fig:parn}
\end{figure}

Similarly, Table~\ref{tab:parSpeedupsSq} reports average performance results for random square matrices of varying sizes. The scalability depicted in Fig.~\ref{fig:parsq:speedup} and the corresponding time-iteration relationship in Fig.~\ref{fig:parsq:it} confirm the same speedup behavior and parallel performance characteristics. These consistent trends reinforce the impact of mandatory global synchronization steps limiting scalability beyond $4$ processes, while lazy approximations help in the reduction of iteration counts and improve performance up to that point.

\begin{table}[!ht]
    \centering
    \resizebox{\textwidth}{!}{%
    \begin{tabular}{c c c c c c c}
    \hline
        \multicolumn{2}{c}{\textbf{Matrix Size ($\mathbf{m\times m}$)}} & $\mathbf{10000\times 10000}$ & $\mathbf{20000\times 20000}$ & $\mathbf{30000\times 30000}$ & $\mathbf{40000\times 40000}$ & $\mathbf{50000\times 50000}$ \\ \hline
        \multirow{3}{*}{$1$ Process} & Iterations & $357.6$ & $383.5$ & $373.8$ & $362.8$ & $343.4$ \\
                              & Time (s) & $18.8422$ & $54.9512$ & $122.2234$ & $197.5582$ & $285.6750$ \\
                              & Speedups & $1.000\times$ & $1.000\times$ & $1.000\times$ & $1.000\times$ & $1.000\times$ \\ \hline
        \multirow{3}{*}{$2$ Processes} & Iterations & $93.6$ & $102.7$ & $106.0$ & $106.4$ & $103.8$ \\
                              & Time (s) & $5.0166$ & $9.5441$ & $21.2242$ & $32.9088$ & $52.2300$ \\
                              & Speedups & $3.755\times$ & $5.756\times$ & $5.757\times$ & $6.004\times$ & $5.468\times$ \\ \hline
        \multirow{3}{*}{$4$ Processes} & Iterations & $67.2$ & $76.3$ & $83.4$ & $86.0$ & $85.8$ \\
                              & Time (s) & $6.5648$ & $7.1169$ & $12.9122$ & $19.9879$ & $28.8644$ \\
                              & Speedups & $2.871\times$ & $7.719\times$ & $9.465\times$ & $9.880\times$ & $9.897\times$ \\ \hline
        \multirow{3}{*}{$8$ Processes} & Iterations & $127.0$ & $76.5$ & $72.8$ & $79.8$ & $83.4$ \\
                              & Time (s) & $24.4893$ & $9.8432$ & $12.9961$ & $19.5361$ & $28.5555$ \\
                              & Speedups & $0.769\times$ & $5.582\times$ & $9.405\times$ & $10.114\times$ & $10.000\times$ \\ \hline
    \end{tabular}}
    \caption{Summary of average results over 10 runs for random square matrices with $99\%$ sparsity with termination condition: $\mathrm{RSE} \leq 10^{-4}$, and $\eta = 0.1$.}
    \label{tab:parSpeedupsSq}
\end{table}
\begin{figure}[!ht]
    \centering
    \begin{subfigure}[b]{0.49\textwidth}
        \includegraphics[width=\textwidth]{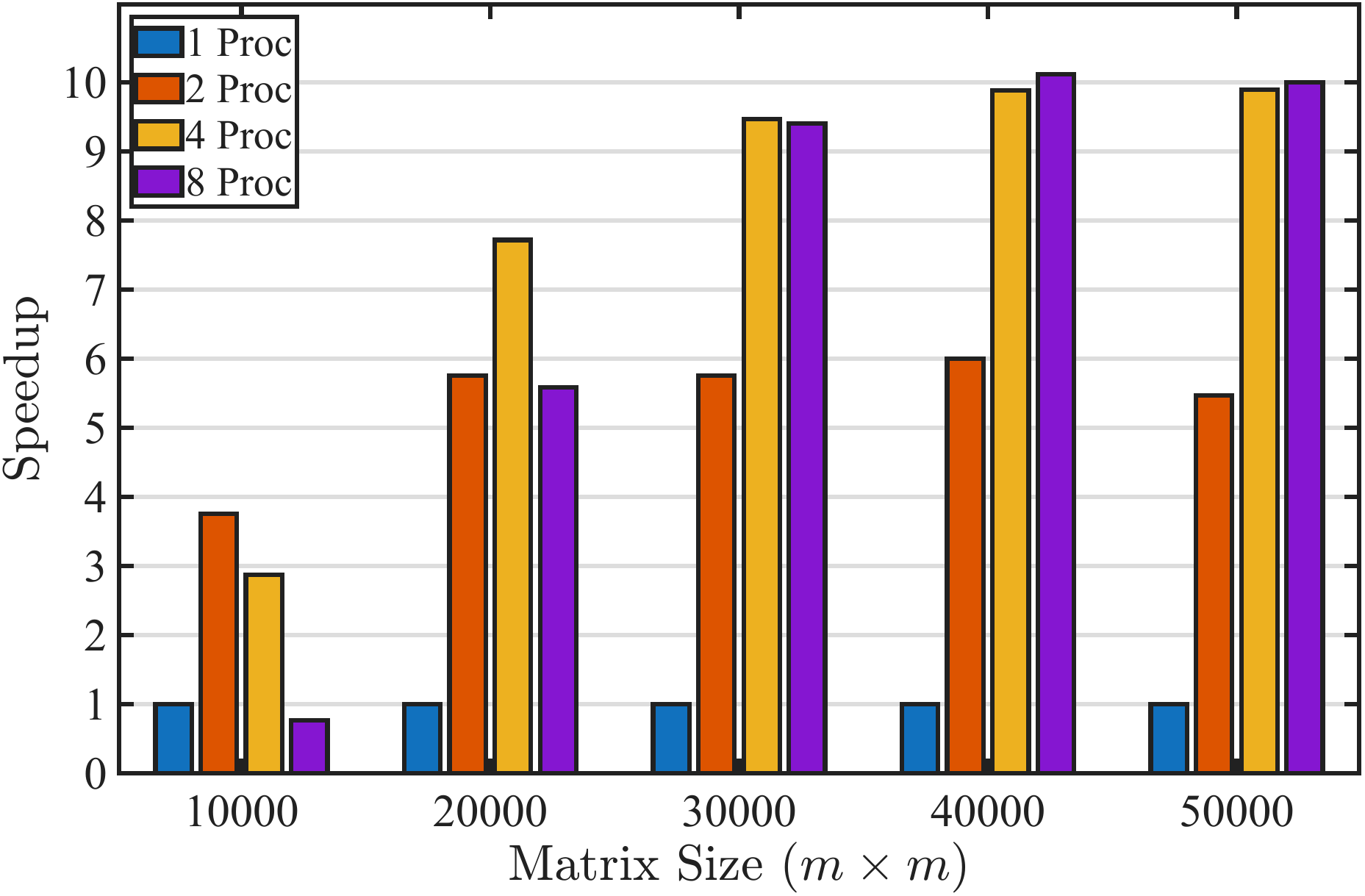}
        \caption{Speedups over different numbers of processes}
        \label{fig:parsq:speedup}
    \end{subfigure}
    \hfill
    \begin{subfigure}[b]{0.49\textwidth}
        \includegraphics[width=\textwidth]{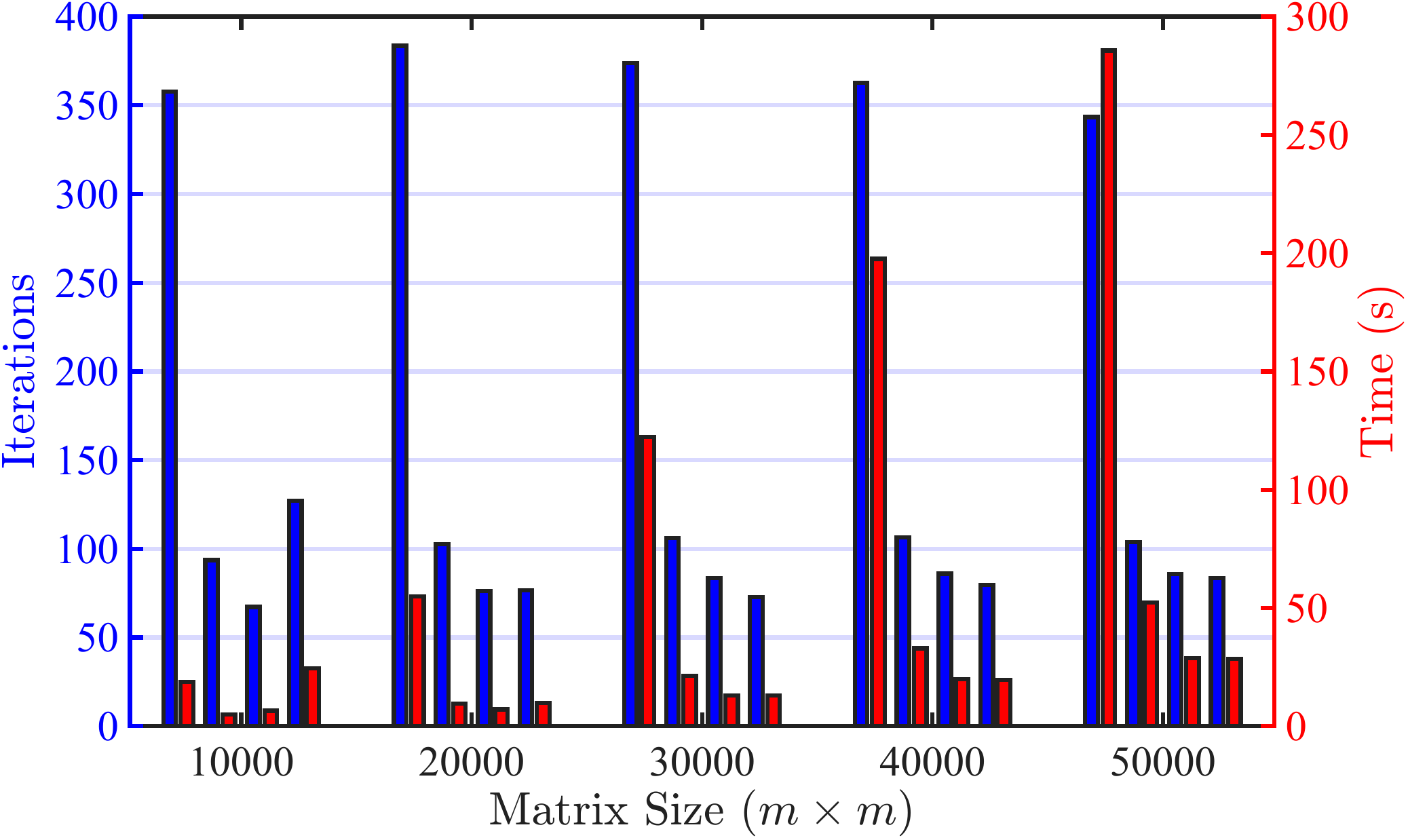}
        \caption{Iterations and Time dependence analysis}
        \label{fig:parsq:it}
    \end{subfigure}
    \caption{Time, Iterations, and Speedups comparisons using results given in Table~\ref{tab:parSpeedupsSq}.}
    \label{fig:parsq}
\end{figure}

\section{Applications}\label{sec:applications}

\subsection{Image Deblurring}
One of the applications of solving linear systems of equations is image deblurring, which is one of the classic problems. Image deblurring can be solved by simultaneously solving multiple large‑scale linear systems arising from each color channel. Given a color image $\mathcal{X} \in \mathbb{R}^{m \times n \times 3}$, we vectorize each of the three channels to obtain systems of the form,
\begin{equation}
    Ax_i = b_i\qquad \text{where, } A\in\mathbb{R}^{mn\times mn}, x_i, b_i\in\mathbb{R}^{mn\times 1} \text{ for } i = 1,2,3,
\end{equation}

where $b_i$ denotes the blurred observations, which are obtained using $Ax_i$ and $x_i$, the corresponding true channel to be recovered. Here, $A$ is assumed to be a banded Toeplitz matrix defined by a one-dimensional Gaussian point-spread function of standard deviation $\sigma$ and half‑width $r$ whose entries are defined as,
\begin{equation}\label{eq:toeplitz}
    A_{ij} =
    \begin{cases}
        \frac{1}{\sigma \sqrt{2\pi}} \exp{\left(-\frac{(i-j)^2}{2\sigma^2}\right)}, & |i-j| \leq r,\\
        0, & \text{otherwise},
    \end{cases}
\end{equation}

with $\sigma=20$ and $r=20$ for all channels. To demonstrate the effectiveness of RGDBEK in real-world scenarios, we conduct experiments on images of size $256\times256\times3$, for which $A$ has dimensions $65536\times65536$. Table~\ref{tab:image} summarizes the Peak Signal‑to‑Noise Ratio (PSNR), Structural Similarity Index (SSIM), and average relative residual over all channels, achieved on three test images (``Cat'', ``Monalisa'', and ``Anime''). Fig.~\ref{fig1:cat}, \ref{fig2:monalisa}, and \ref{fig3:anime} visually compare, for each test image, the original, blurred, and deblurred results, highlighting the restoration of fine details and color fidelity.

\begin{table}[!ht]
    \centering
    \begin{tabular}{c c c c c}\hline
        \textbf{Image} & \textbf{Pixels} & \textbf{PSNR} & \textbf{SSIM} & \textbf{Avg. RSE ($\mathbf{\times10^{-12}}$)} \\\hline
        Cat (Fig.~\ref{fig1:cat}) & $256\times 256\times 3$ & $46.59$ & $0.9982$ & $4.303157$ \\
        Monalisa (Fig.~\ref{fig2:monalisa}) & $256\times 256\times 3$ & $50.18$ & $0.9980$ & $2.474059$ \\
        Anime (Fig.~\ref{fig3:anime}) & $256\times 256\times 3$ & $41.21$ & $0.9878$ & $6.247516$ \\\hline
    \end{tabular}
    \caption{Deblurring of images}
    \label{tab:image}
\end{table}

\begin{figure}[!ht]
    \centering
    \begin{subfigure}[b]{0.25\textwidth}
        \includegraphics[width=\textwidth]{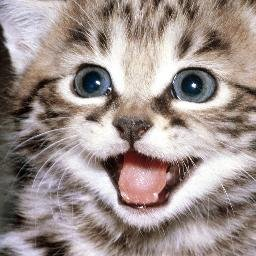}
        \caption{Original}
        \label{fig1:original}
    \end{subfigure}
    \hfill
    \begin{subfigure}[b]{0.25\textwidth}
        \includegraphics[width=\textwidth]{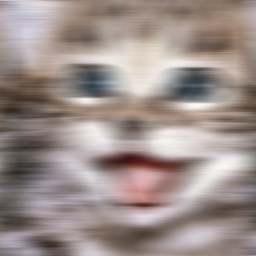}
        \caption{Blurred}
        \label{fig1:blur}
    \end{subfigure}
    \hfill
    \begin{subfigure}[b]{0.25\textwidth}
        \includegraphics[width=\textwidth]{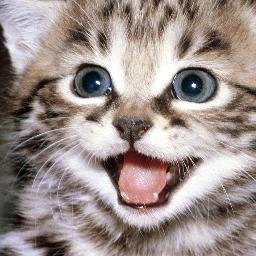}
        \caption{Restored}
        \label{fig1:recons}
    \end{subfigure}
    \caption{Cat image of size $256\times 256\times 3$.}
    \label{fig1:cat}
\vspace{.8cm}
    \begin{subfigure}[b]{0.25\textwidth}
        \includegraphics[width=\textwidth]{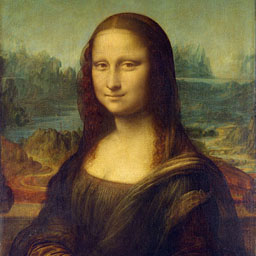}
        \caption{Original}
        \label{fig2:original}
    \end{subfigure}
    \hfill
    \begin{subfigure}[b]{0.25\textwidth}
        \includegraphics[width=\textwidth]{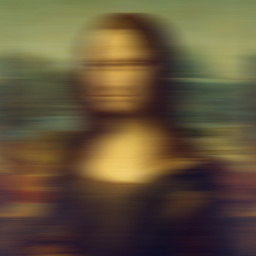}
        \caption{Blurred}
        \label{fig2:blur}
    \end{subfigure}
    \hfill
    \begin{subfigure}[b]{0.25\textwidth}
        \includegraphics[width=\textwidth]{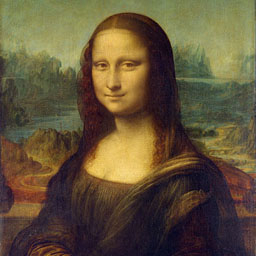}
        \caption{Restored}
        \label{fig2:recons}
    \end{subfigure}
    \caption{Monalisa image of size $256\times 256\times 3$.}
    \label{fig2:monalisa}
\vspace{.8cm}
    \begin{subfigure}[b]{0.25\textwidth}
        \includegraphics[width=\textwidth]{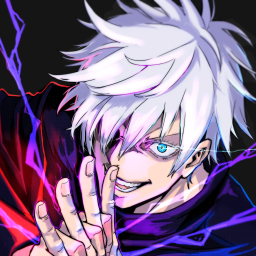}
        \caption{Original}
        \label{fig3:original}
    \end{subfigure}
    \hfill
    \begin{subfigure}[b]{0.25\textwidth}
        \includegraphics[width=\textwidth]{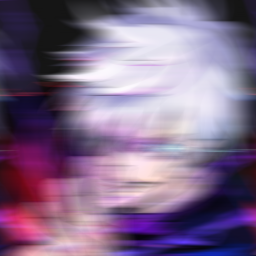}
        \caption{Blurred}
        \label{fig3:blur}
    \end{subfigure}
    \hfill
    \begin{subfigure}[b]{0.25\textwidth}
        \includegraphics[width=\textwidth]{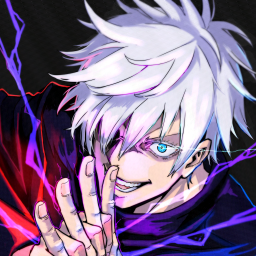}
        \caption{Restored}
        \label{fig3:recons}
    \end{subfigure}
    \caption{Anime image of size $256\times 256\times 3$.}
    \label{fig3:anime}
\end{figure}

\subsection{Finite Element Method}
Large, sparse linear systems of equations arise in solving partial differential equations (PDEs) using the finite difference (FD) or finite element method (FEM). Poisson and Helmholtz equations are fundamental elliptical PDEs widely used in physics and engineering to model steady-state diffusion and wave phenomena, respectively. Discretizing these PDEs using the FEM leads to linear systems of the form, $A \mathbf{u} = \mathbf{b},$
where the system matrix $A \in \mathbb{R}^{N \times N}$ is the matrix arising from the weak formulation of these PDEs over a discretized mesh with $N$ degrees of freedom. $\mathbf{u}$ contains values of the approximate solution, and $\mathbf{b}$ is the discretized source. 

\subsubsection{Poisson equation} 
The Poisson equation is given by,
\begin{equation}
    -\Delta u = f \quad \text{in } \Omega, \quad u=0\text{ on } \partial \Omega,
\end{equation}

which is discretized using linear basis functions on triangles using Delaunay triangulation. The exact solution used for validation is,
$$u(x,y) = \sin(\pi x) \sin(\pi y),$$

with source term,
$$f(x,y) = 2 \pi^2 \sin(\pi x) \sin(\pi y).$$

For experimentation, we have assumed mesh domain as $\Omega = [0,1]\times [0,1]$, which forms uniform grid with $n_x = 25, n_y = 25$ resulting in $N = n_xn_y = 625$ nodes. The experiment is run for $\eta = 0.5$ and maximum $10000$ iterations with tolerance $RSE\leq 10^{-6}$.

\begin{figure}[!ht]
    \centering
    \begin{subfigure}[b]{0.3\textwidth}
        \includegraphics[width=\textwidth]{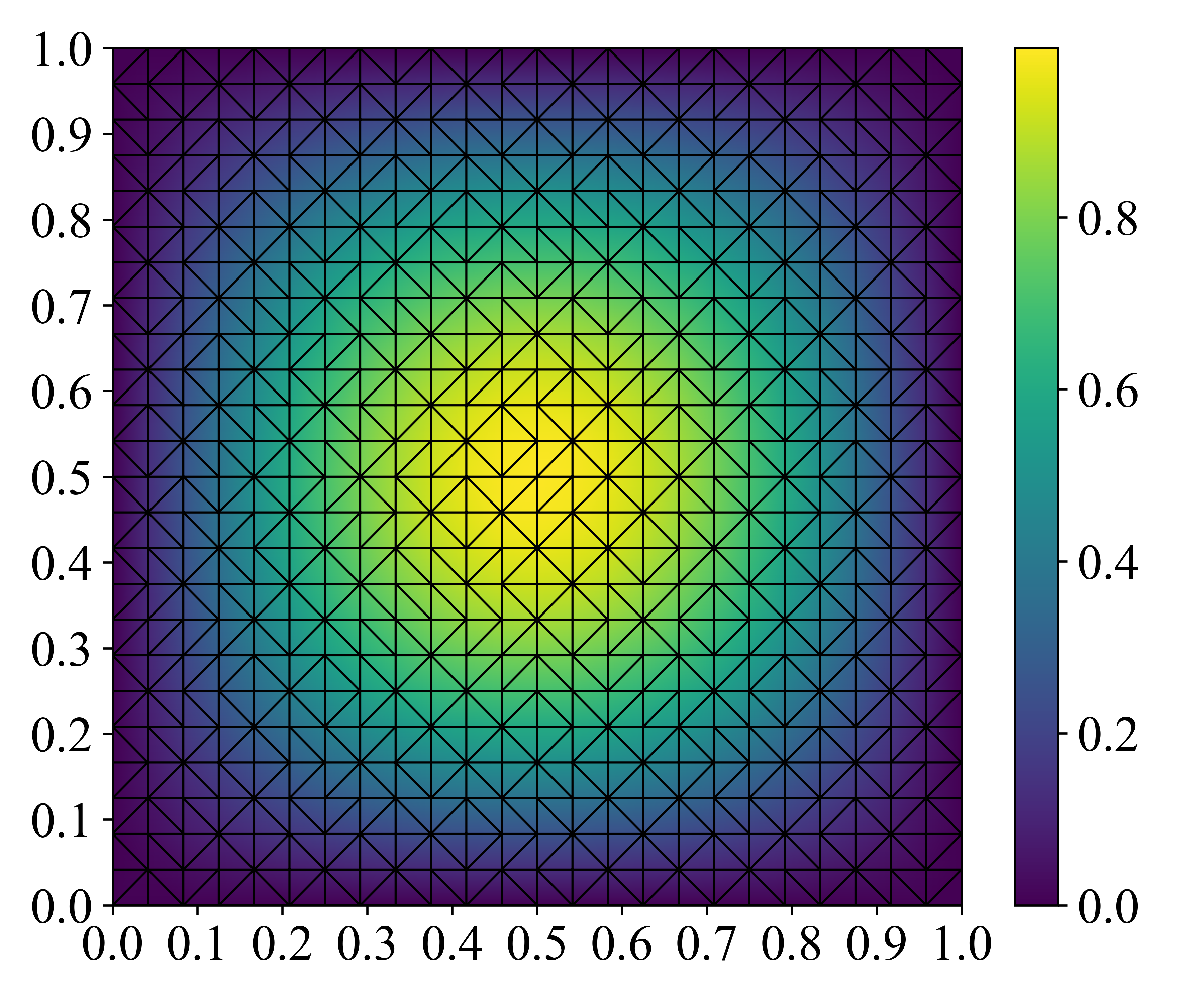}
        \caption{FEM Solution}
    \end{subfigure}
    \hfill
    \begin{subfigure}[b]{0.3\textwidth}
        \includegraphics[width=\textwidth]{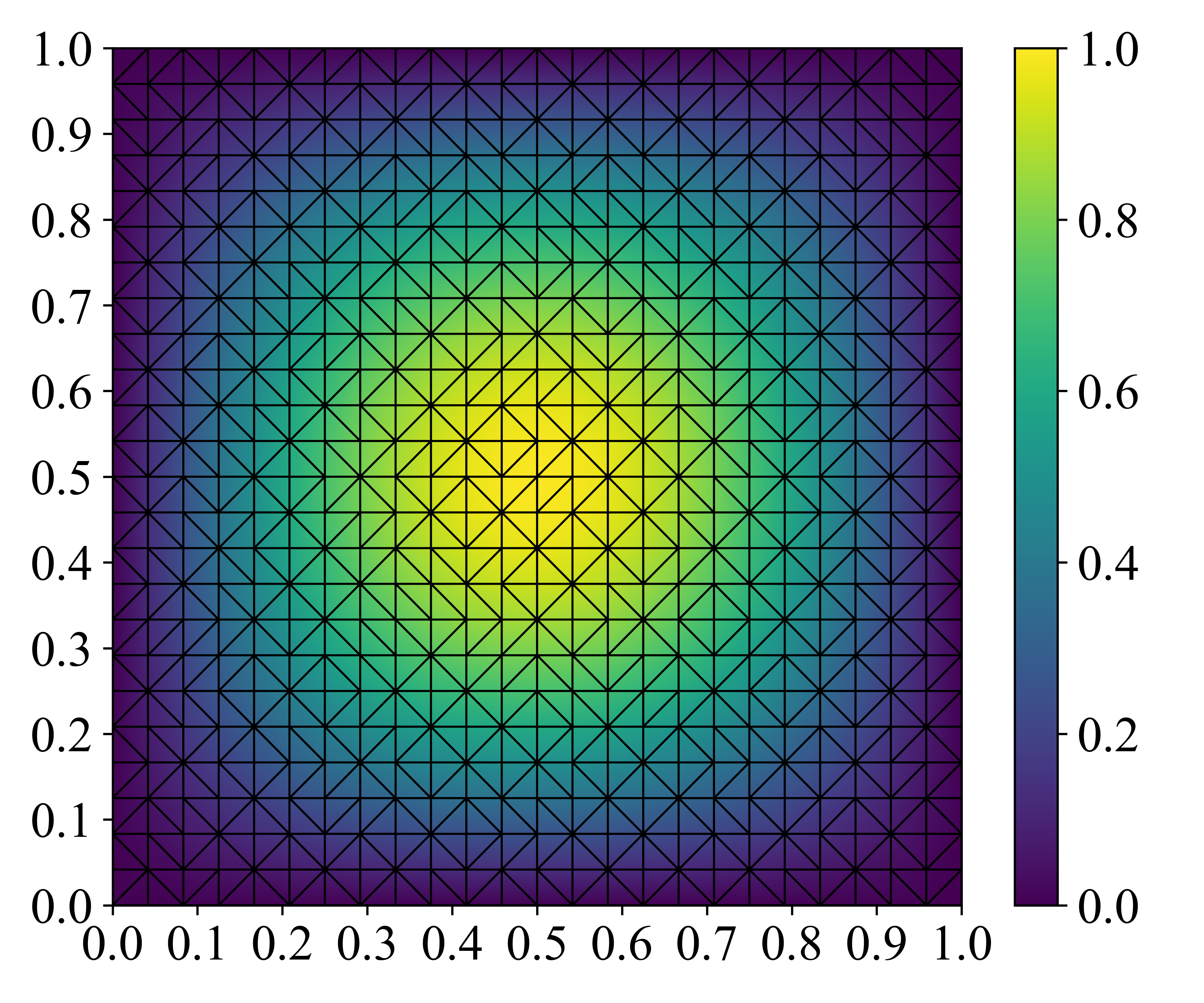}
        \caption{Exact Solution}
    \end{subfigure}
    \hfill
    \begin{subfigure}[b]{0.3\textwidth}
        \includegraphics[width=\textwidth]{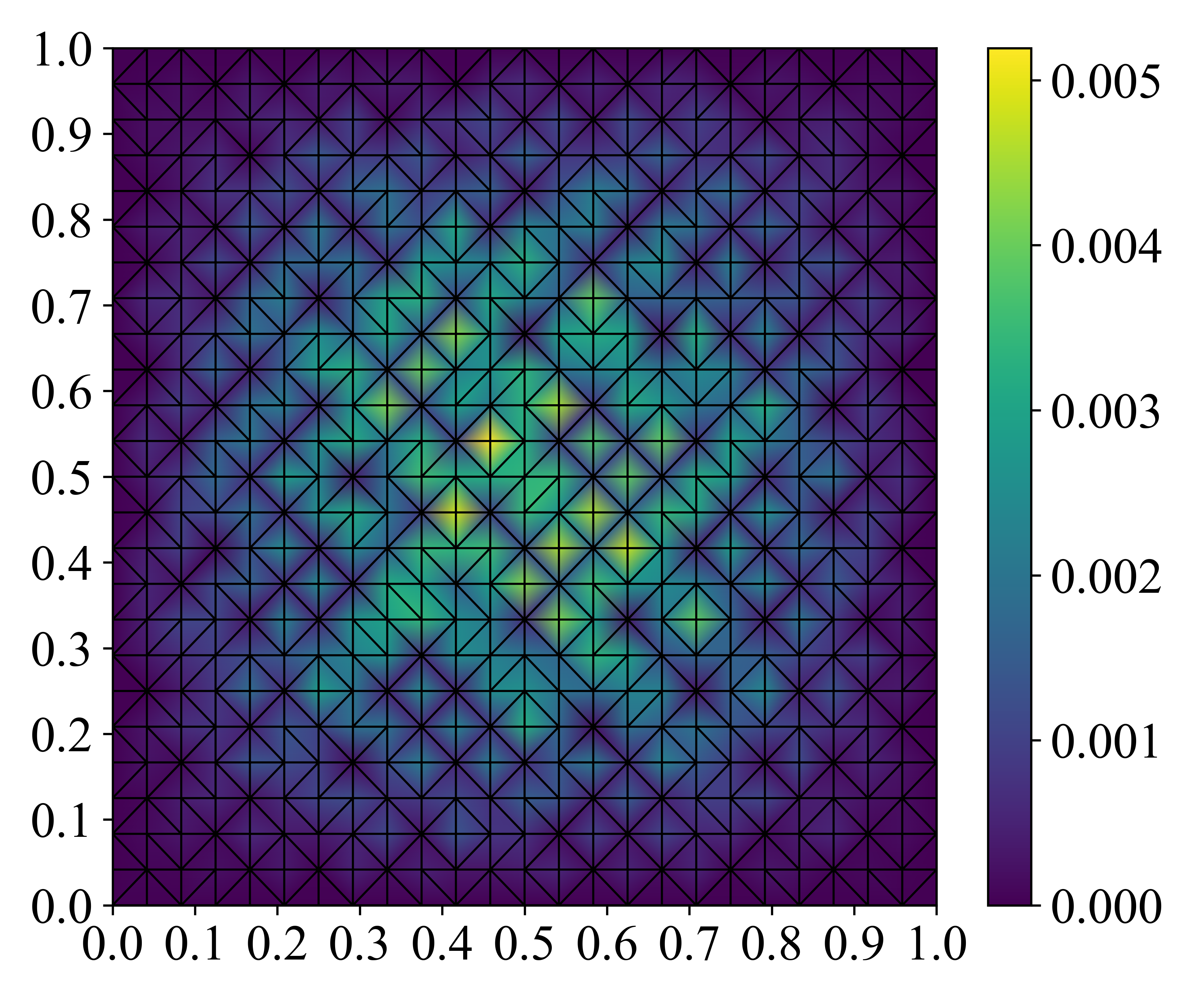}
        \caption{Absolute Error}
    \end{subfigure}
    \caption{Solution of Poisson equation using $1152$ Triangular elements}
    \label{fig:pde1}
\end{figure}
 
\subsubsection{Helmholtz equation} 
The Helmholtz PDE with wavenumber $k$,
\begin{equation}
    -\Delta u - k^2 u = f \quad \text{in } \Omega, \quad u = 0 \text{ on } \partial \Omega,
\end{equation}

is discretized similarly. The exact solution used for validation is,
$$u(x,y) = \sin(4\pi x) \sin(4\pi y),$$

with source term,
$$f(x,y) = \left(2(4 \pi)^2 - k^2 \right) \sin(4\pi x) \sin(4\pi y).$$

For experimentation, we have assumed a wave number $k=25$, mesh domain as $\Omega = [0,1]\times [0,1]$, which forms a uniform grid with $n_x = 50, n_y = 50$ resulting in $N = n_xn_y = 2500$ nodes. The experiment is run for $\eta = 0.5$ and maximum $10000$ iterations with tolerance $RSE\leq 10^{-6}$.

\begin{figure}[!ht]
    \centering
    \begin{subfigure}[b]{0.3\textwidth}
        \includegraphics[width=\textwidth]{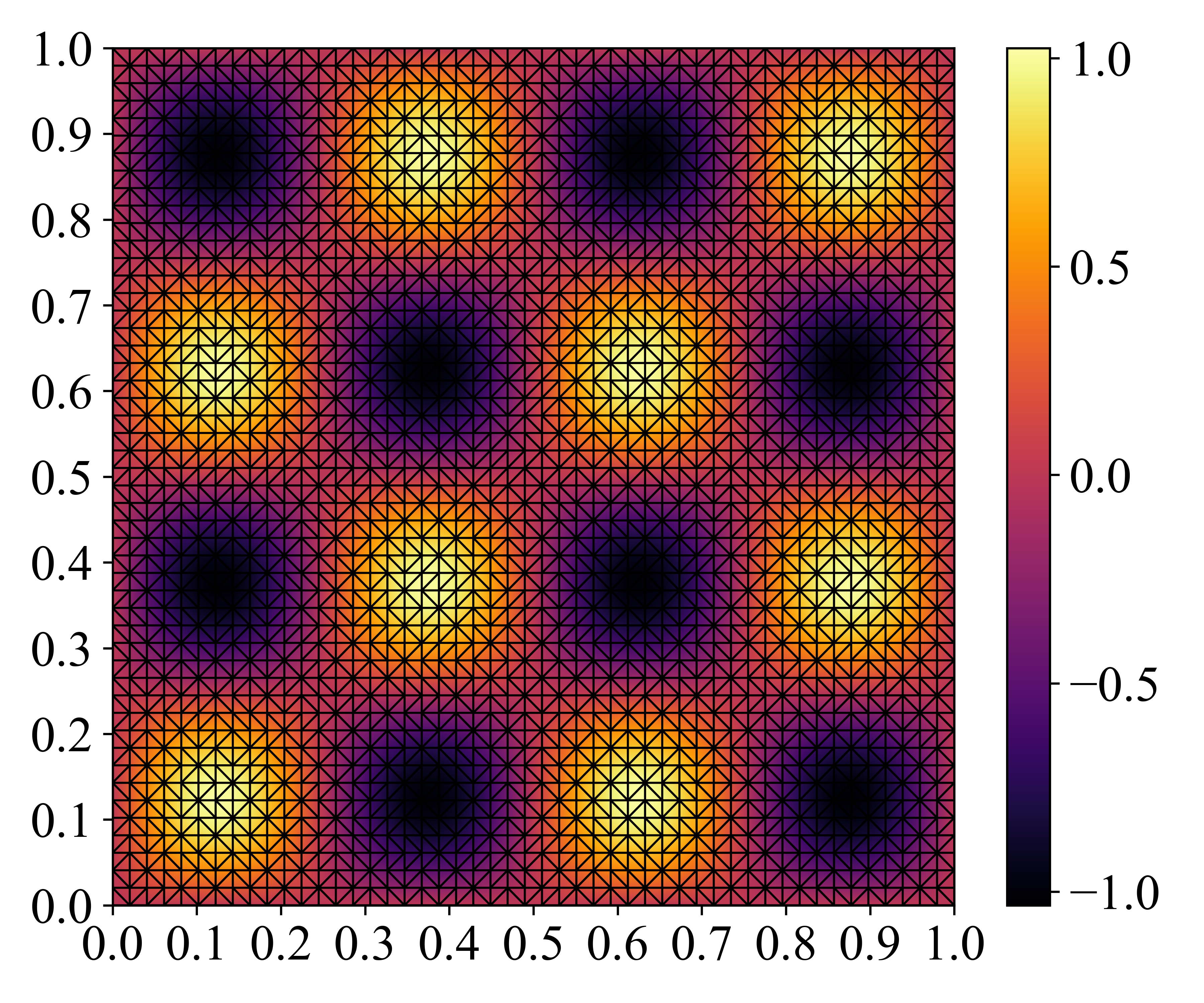}
        \caption{FEM Solution}
    \end{subfigure}
    \hfill
    \begin{subfigure}[b]{0.3\textwidth}
        \includegraphics[width=\textwidth]{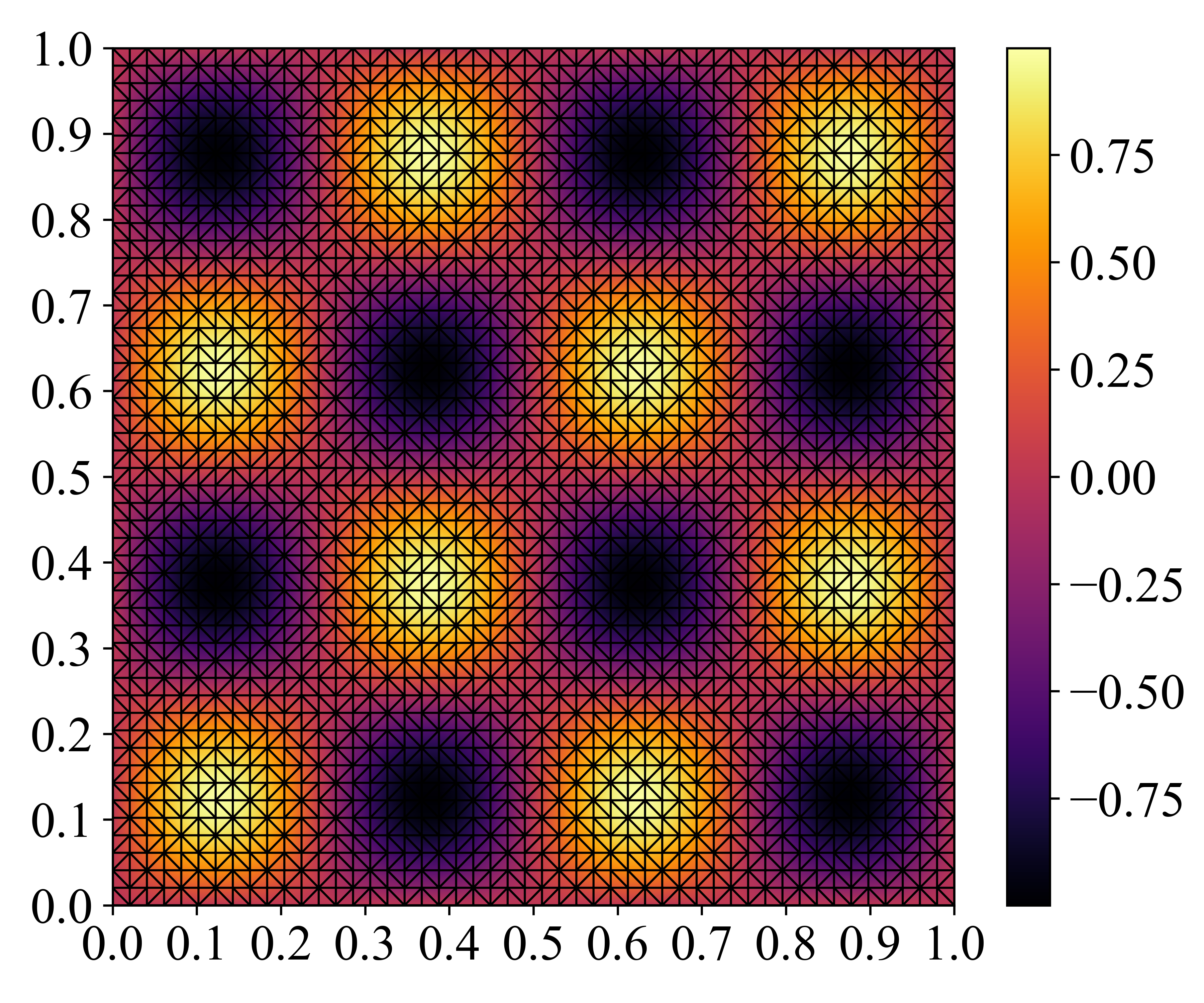}
        \caption{Exact Solution}
    \end{subfigure}
    \hfill
    \begin{subfigure}[b]{0.3\textwidth}
        \includegraphics[width=\textwidth]{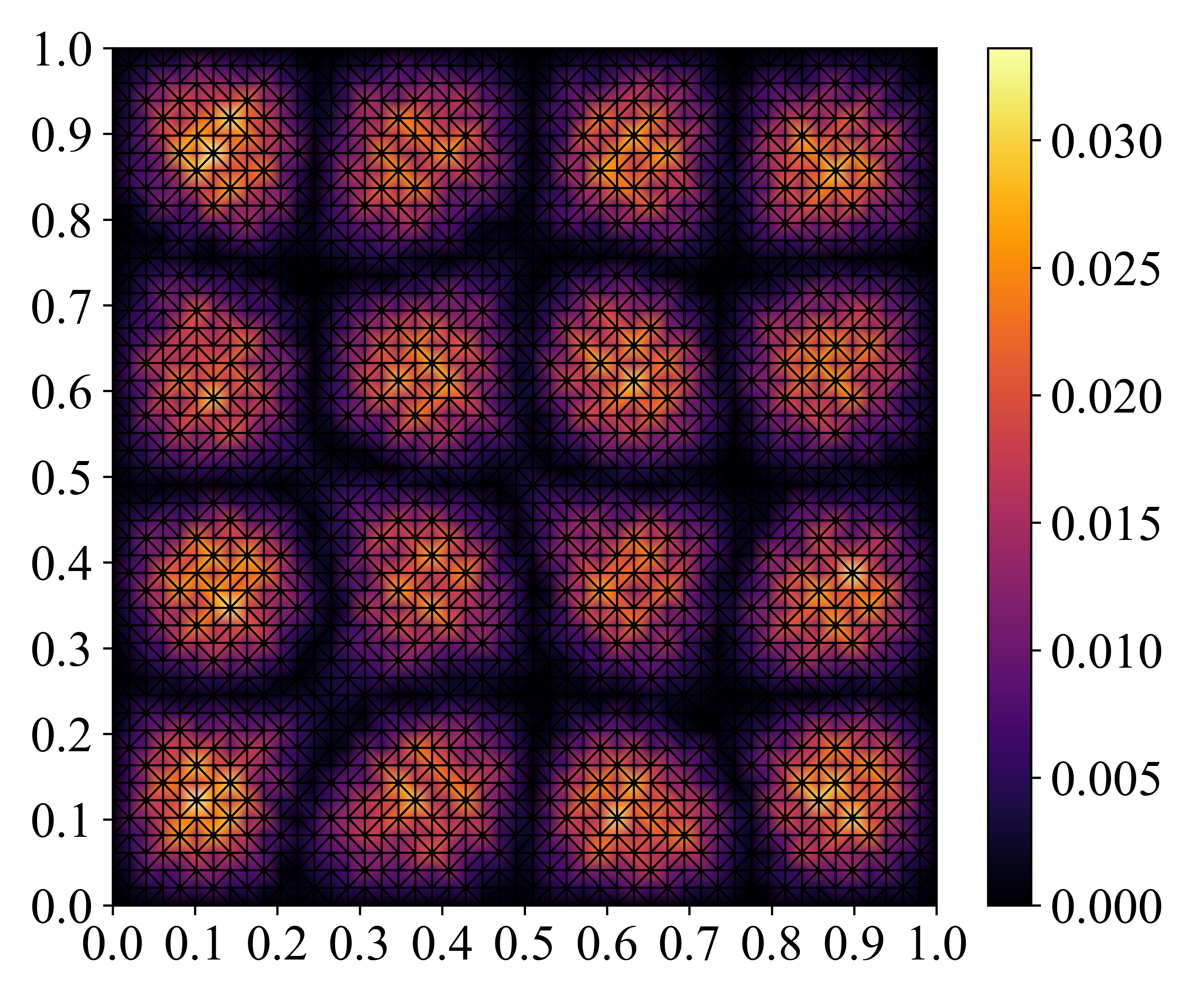}
        \caption{Absolute Error}
    \end{subfigure}
    \caption{Solution of Helmholtz equation using $4802$ Triangular elements}
    \label{fig:pde2}
\end{figure}

\subsubsection{Results}
Table~\ref{tab:poisson_helmholtz} summarizes the performance results for both Poisson and Helmholtz PDEs discretized using the FEM on triangular meshes. From the table, it is evident that the FEM system matrices are highly sparse and exhibit very large condition numbers. For a RSE threshold of $10^{-6}$, the relative $L_2$ error is computed as $3.24 \times 10^{-3}$ for the Poisson equation and $2.33 \times 10^{-2}$ for the Helmholtz equation. The FEM reconstructions, along with the absolute error distributions across the computational domain, are depicted in Fig.~\ref{fig:pde1} and Fig.~\ref{fig:pde2}.

\begin{table}[!ht]
\centering
\begin{tabular}{c c c }
\hline
\textbf{Metric} & \textbf{Poisson} & \textbf{Helmholtz} \\\hline
Number of elements & $1152$ & $4802$ \\
Sparsity of $A$ & $99.32\%$ & $99.74\%$ \\
Condition number of $A$ & $2.327776 \times 10^{2}$ & $3.683559 \times 10^{3}$ \\
Frobenius Norm of $A$ & $1.028786 \times 10^{2}$ & $2.100564 \times 10^{2}$ \\
RSE of system & $9.359952 \times 10^{-7}$ & $9.970865 \times 10^{-7}$ \\
Relative $L_2$ Error & $3.238629 \times 10^{-3}$ & $2.329342 \times 10^{-2}$ \\
Std. Deviation of Error & $1.090497 \times 10^{-3}$ & $7.261207 \times 10^{-3}$ \\\hline
\end{tabular}
\caption{Performances for solving Poisson and Helmholtz equations using FEM Methods}
\label{tab:poisson_helmholtz}
\end{table}

\subsection{Signal Recovery in Mathematical Modelling}
Population modeling is considered one of the most difficult problems in mathematical models due to noise incurred by natural factors, including earthquakes and tsunamis. However, we can derive an approximate solution to this problem using past values by forming a linear system, which denoises the values to give us approximate forward time values. To demonstrate this, we have used the Predator Prey Scavenger model~\cite{panchal2025predator}, which is governed by the following set of ordinary differential equations,
\begin{equation}\label{eq:predpreyscav}
\begin{split}    
\frac{dx}{dt} &= rx\left(1 - \frac{x}{k}\right) - \frac{ax^2 y}{1 + a_0 x^2} - \frac{bx^2 z}{1 + b_0 x^2},\\
\frac{dy}{dt} &= \frac{dx^2 y}{1 + a_0 x^2} + \frac{fz^2 y}{1 + i_0 z^2} - ey,\\
\frac{dz}{dt} &= \frac{gx^2 z}{1 + b_0 x^2} + hyz - \frac{iyz^2}{1 + i_0 z^2} - jz, 
\end{split}
\end{equation}

with all real positive parameter values and strictly positive initial populations. We define the populations at a particular time $t$ as the convolution of delayed noisy values and filter coefficients,
\begin{equation*}
x(t) = \sum_{i=0}^{n} m_x(t-i)*c_x(i), \quad y(t) = \sum_{i=0}^{n} m_y(t-i)*c_y(i), \quad z(t) = \sum_{i=0}^{n} m_z(t-i)*c_z(i),
\end{equation*}

where, $m_x(t) = x(t-k) + \mathbf{n}(t),\; m_y = y(t-k) + \mathbf{n}(t),\; m_z = z(t-k) + \mathbf{n}(t),\; \mathbf{n}(t)\sim\mathcal{N}(\mu,\sigma^2)$ are the noisy delayed population values with delay parameter $k\in{0,1,2,\dots, n}$, and $c_x, c_y$ and $c_z$ are the respective filter coefficients. Then, the system can be seen as a system of linear equations given by,
\begin{equation*}
\mathbf{M_v} = 
\begin{bmatrix}
m_v(t) & m_v(t-1) & \cdots & m_v(t-n) \\
m_v(t+1) & m_v(t) & \cdots & m_v(t-n+1) \\
\vdots & \vdots & \ddots & \vdots \\
m_v(t+n) & m_v(t+n-1) & \cdots & m_v(t) \\
\end{bmatrix}, \quad\mathbf{c_v} = \begin{bmatrix}
c_v(0)\\c_v(1)\\ \vdots \\c_v(n)
\end{bmatrix},
\quad\mathbf{v} = \begin{bmatrix}
v(t)\\ v(t + 1) \\\vdots \\ v(t + n)
\end{bmatrix},
\end{equation*}

for $v\in\{x,y,z\}$, which can be written as,
\begin{equation*}
    \mathbf{M_v} \mathbf{c_v} = \mathbf{v},
\end{equation*}

where, $\mathbf{M_v}$ is the noisy delayed population matrix, $\mathbf{c_v}$ is the filter coefficient and $\mathbf{v}$ is the actual population for prey ($x(t)$), predators ($y(t)$) and scavengers ($z(t)$). We initially compute the filter coefficient using the proposed RGDBEK algorithm. Then, filter coefficients are used to predict the population from noisy populations. Here, we have assumed that the noise distribution is the same across the spread of the populations. 

For the experiment, Table~\ref{tab:param} shows the parameters of Eq.~\eqref{eq:predpreyscav} we have assumed with initial conditions $x(0) = 4$, $y(0) = 3$ and $z(0) = 2$. We simulated the solution from $T = 0$ to $T=200$ with $0.1$ time step. We assumed the Gaussian noise with mean $\mu=0$ and standard deviation $\sigma=3$ along with a delay parameter $k=1$. For RGDBEK, we fixed the maximum number of iterations to be $10000$ with $\eta=0.5$ and tolerance $RSE\leq 10^{-10}$.

\begin{table}[!ht]
\centering
\begin{tabular}{c c c c c c c c c c c c c c c}
\hline
Parameter & $r$ & $k$ & $a$ & $a_0$ & $b$ & $b_0$ & $d$ & $e$ & $f$ & $g$ & $h$ & $i$ & $i_0$ & $j$ \\
\hline
Value & 0.5 & 100 & 0.5 & 0.25 & 0.5 & 0.25 & 0.5 & 1 & 0.1 & 0.5 & 0.1 & 0.1 & 0.25 & 1 \\
\hline
\end{tabular}
\caption{Numerical values of the parameters of Eq.~\eqref{eq:predpreyscav} used.}
\label{tab:param}
\end{table}

\begin{figure}[!ht]
    \centering
    \begin{subfigure}[b]{0.49\textwidth}
        \includegraphics[width=\textwidth]{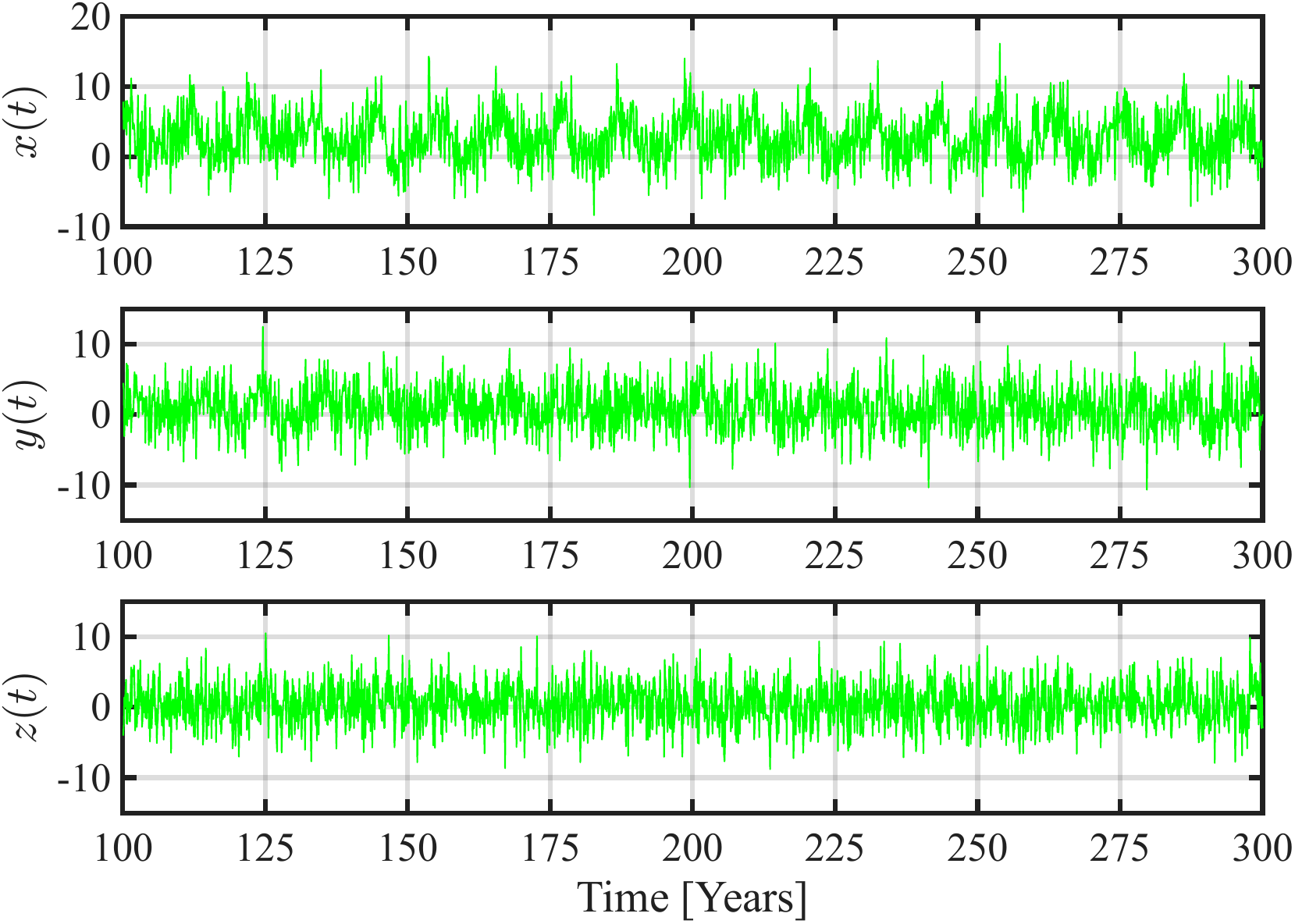}
        \caption{Noisy Delayed Populations}
        \label{fig:noisy}
    \end{subfigure}
    \hfill
    \begin{subfigure}[b]{0.49\textwidth}
        \includegraphics[width=\textwidth]{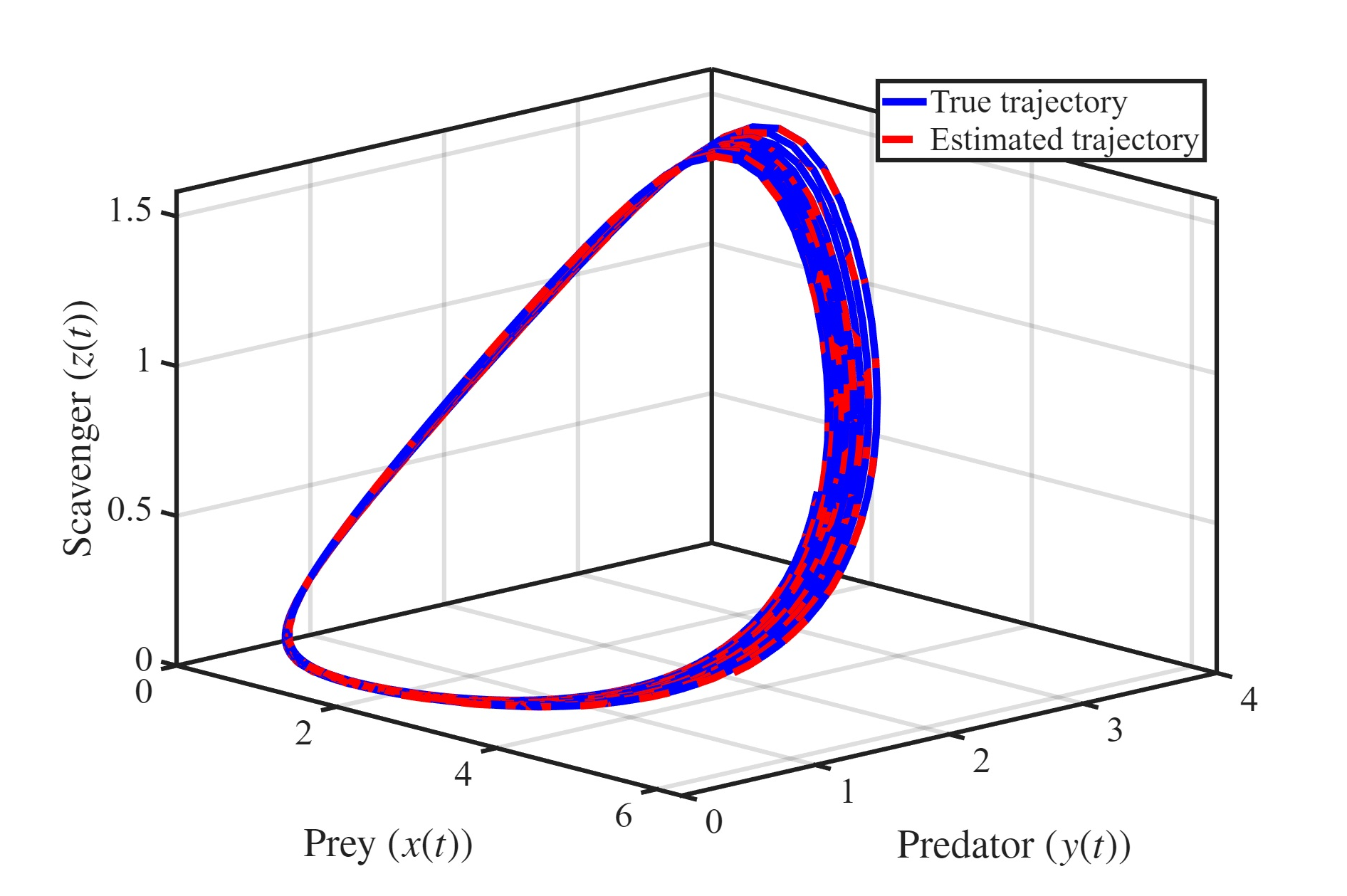}
        \caption{Phase Portrait of the System}
        \label{fig:phase}
    \end{subfigure}
    \caption{Noisy delayed populations and Phase portrait of the true and estimated system.}
    \label{fig:signal}
\end{figure}

\begin{figure}[!ht]
    \centering
    \begin{subfigure}[b]{0.49\textwidth}
        \includegraphics[width=\textwidth]{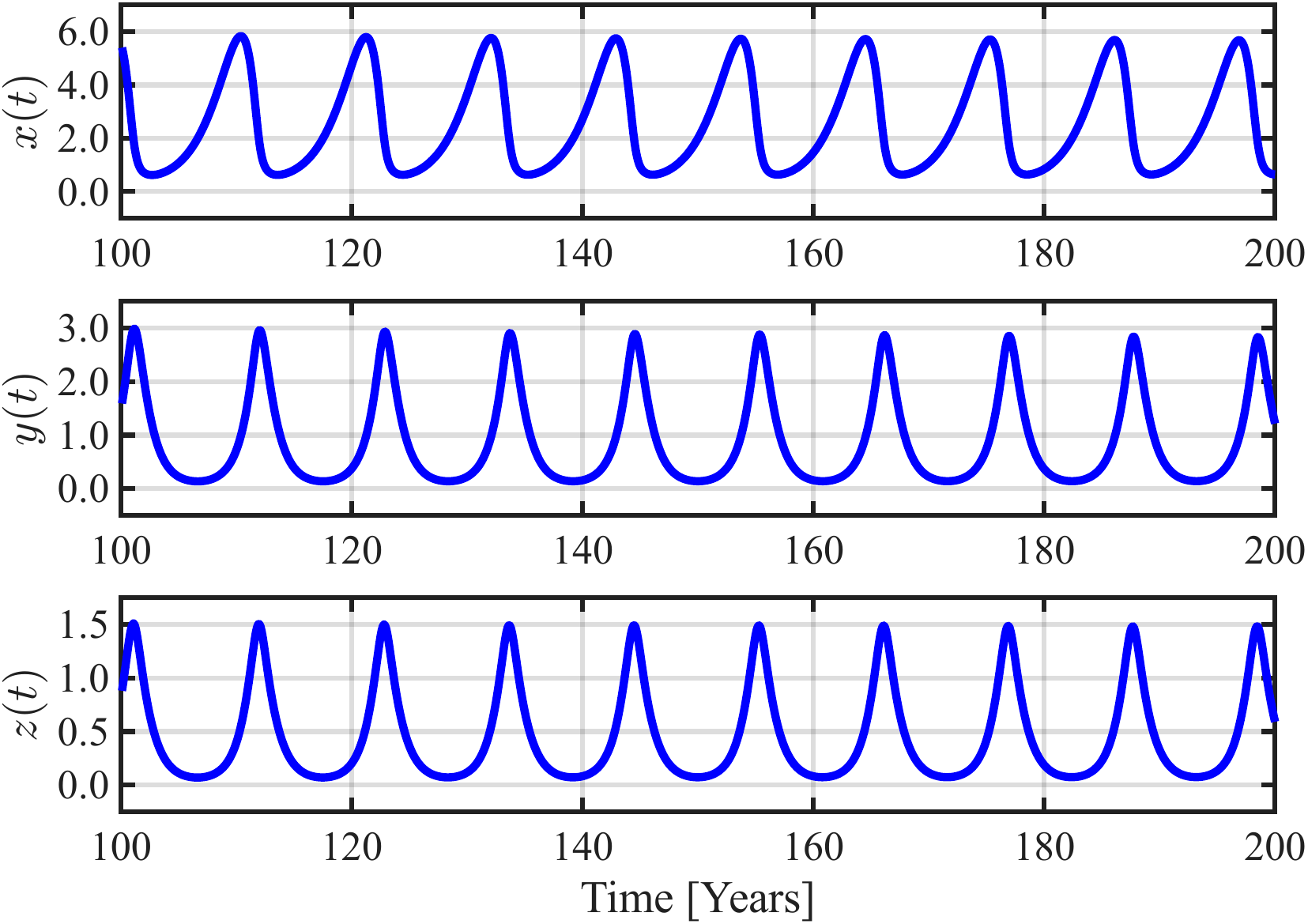}
        \caption{True Populations}
        \label{fig:true}
    \end{subfigure}
    \hfill
    \begin{subfigure}[b]{0.49\textwidth}
        \includegraphics[width=\textwidth]{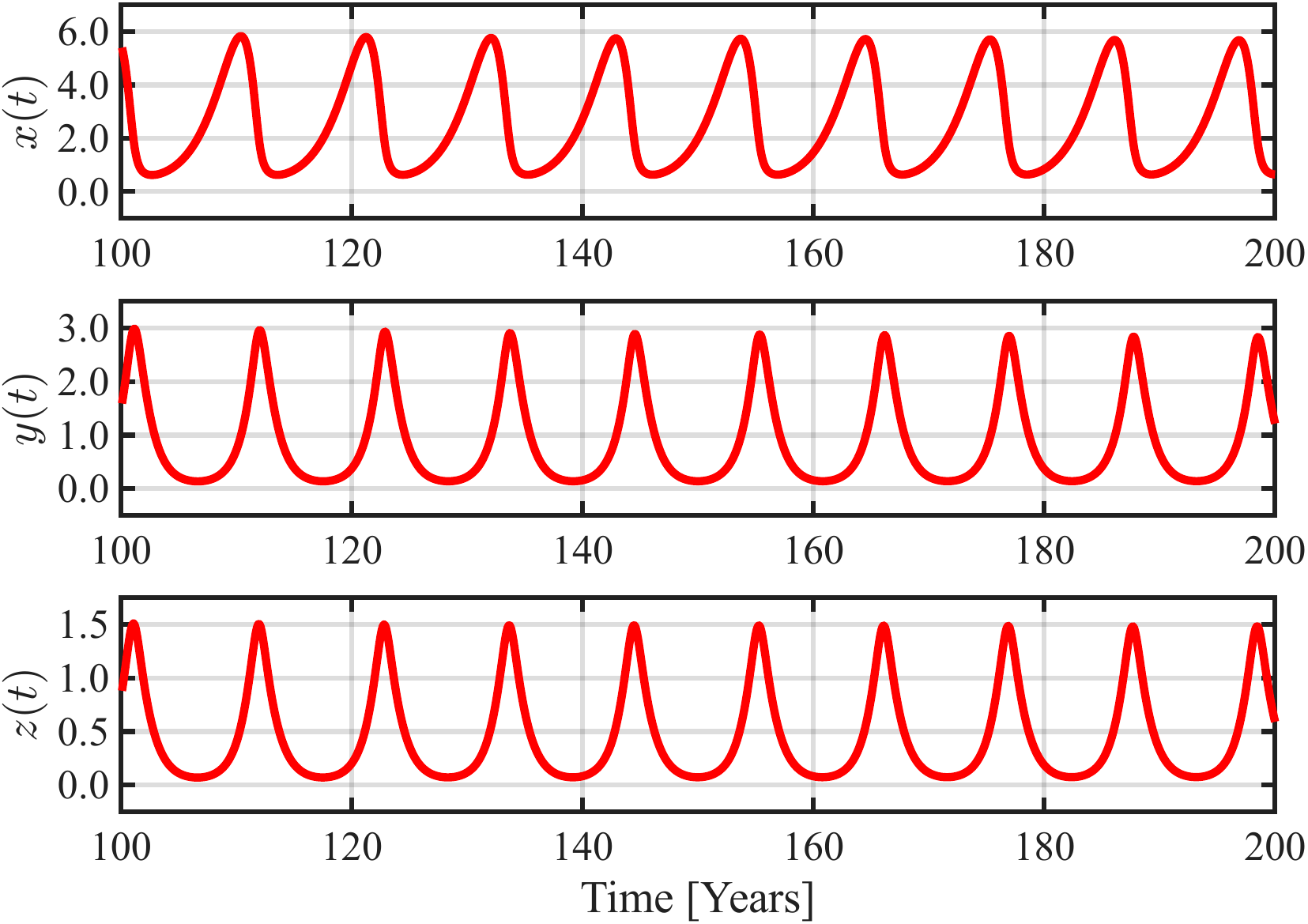}
        \caption{Estimated Populations}
        \label{fig:predict}
    \end{subfigure}
    \caption{Comparison of true and estimated populations.}
    \label{fig:signalPredict}
\end{figure}

Fig.~\ref{fig:noisy} shows the noisy delayed population data considered for denoising, which is created according to the experimental setup described above. It is evident from Fig.~\ref{fig:true}, \ref{fig:predict}, and \ref{fig:phase} that the denoising performance is excellent. Specifically, Fig.~\ref{fig:true} depicts the true population values, Fig.~\ref{fig:predict} illustrates the predicted population values after denoising, and Fig.~\ref{fig:phase} presents the phase portrait comparing the true and predicted population trajectories. To quantitatively assess the reconstruction accuracy, we computed the Frobenius norm of the reconstruction error, which yields a small value of $0.0006123908$. This low error confirms the high-quality reconstruction achieved by the RGDBEK algorithm, demonstrating its promising potential for future applications in signal processing and population modeling.

\section{Conclusions}\label{sec:conclusions}
In this paper, we propose the RGDBEK algorithm, a novel iterative method for solving large-scale linear systems. The RGDBEK algorithm introduces a randomized selection strategy for column and row blocks based on residual-derived probability distributions, effectively mitigating the conventional seesaw effect and improving convergence robustness. Theoretical analysis established the linear convergence of RGDBEK under standard assumptions. We performed extensive numerical experiments on both synthetic random matrices and real-world sparse matrices from the \texttt{SuiteSparse} collection, which demonstrate that RGDBEK consistently outperforms existing state-of-the-art Kaczmarz variants, including GRK, FDBK, FGBK, and GDBEK, in terms of iteration counts and computational time. We also developed a hybrid parallel implementation that leverages CPU-GPU architectures, further enhancing scalability and performance on large-scale sparse problems while efficiently managing sparse matrix-vector multiplications via the BSAS storage format.

Applications in finite element discretizations of Poisson and Helmholtz PDEs, image deblurring, and noisy population modeling illustrate the practical significance and versatility of the proposed algorithm. Overall, RGDBEK provides a powerful and scalable framework for solving challenging large-scale linear systems arising in scientific computing and engineering applications. Future work includes extending the RGDBEK algorithm to tensor systems, analyzing the impact of the parameter $\eta$ in parallel computational settings, and reducing communication overheads through the development of a novel algorithm for efficient column projections to improve efficiency and applicability further.

\section*{Declarations}
\subsection*{Ethical approval}
Not applicable
\subsection*{Availability of supporting data}
The authors confirm that the data supporting the findings of this study are available within the article. 
\subsection*{Competing interests}
The author declares no competing interests.
\subsection*{Funding}
Ratikanta Behera is grateful for the support of the Anusandhan National Research Foundation (ANRF), Govt. of India, under Grant No. EEQ/2022/001065. 
 
\section*{ORCID}
Ratikanta Behera\orcidA \href{https://orcid.org/0000-0002-6237-5700}{ \hspace{2mm}\textcolor{lightblue}{ https://orcid.org/0000-0002-6237-5700}}\\
Aneesh Panchal \orcidC \href{https://orcid.org/0009-0005-5727-3901}{ \hspace{2mm}\textcolor{lightblue}{https://orcid.org/0009-0005-5727-3901}} \\

\bibliographystyle{abbrv}
\bibliography{bibliography}

\end{document}